\documentclass[dvips]{article}

\usepackage{epsf}
\usepackage{latexsym}
\usepackage{amsmath}
\usepackage{amssymb}
\RequirePackage{amsthm,amsmath}

\renewcommand{\[}{\begin{eqnarray*}}
\renewcommand{\]}{\end{eqnarray*}}

\newtheorem{thm}{Theorem}[section]

\begin{document}

\title{A refinement of Bennett's inequality with applications to
  portfolio optimization}

\author{Tony Jebara}

\maketitle

\begin{abstract}
  A refinement of Bennett's inequality is introduced which is strictly
  tighter than the classical bound. The new bound establishes the
  convergence of the average of independent random variables to its
  expected value. It also carefully exploits information about the
  potentially heterogeneous mean, variance, and ceiling of each random
  variable. The bound is strictly sharper in the homogeneous setting
  and very often significantly sharper in the heterogeneous
  setting. The improved convergence rates are obtained by leveraging
  Lambert's ${\cal W}$ function. We apply the new bound in a portfolio
  optimization setting to allocate a budget across investments with
  heterogeneous returns.
\end{abstract}

\section{Introduction}

This article studies how the sum of scalar independent random
variables can deviate from its expected value. It extends previous
bounds from Hoeffding \cite{Hoe63} which exploited information about
the range of each random variable. It also outperforms previous bounds
from Bennett \cite{Ben62} which also incorporates information about
the variance of each random variable as well as a single fixed range
on all random variables. The proposed bound simultaneously leverages
information about the heterogeneous means, variances and ranges of the
random variables to obtain better convergence rates in a wider range
of useful settings. These rates give tighter guarantees on the
probability a given portfolio will underperform. This permits improved
portfolio optimization in settings where second-order information is
available while parametric information remains unavailable (as we
shall show in Section~\ref{sct:applications}).

Recall Hoeffding's celebrated inequality \cite{Hoe63}.
\begin{thm}[Hoeffding, 1963]
\label{thm:hoeffding}
Given $X_1,\ldots,X_n$ scalar independent random variables
bounded\footnote{Throughout this article, statements of the form $L_i
  \leq X_i \leq M_i$ will be assumed to be equivalent to $\Pr(X_i \in
  [L_i,M_i])=1.$ Furthermore, we will denote the probability of
  deviating by the shorthand $\Pr$.} as $L_i \leq X_i \leq M_i$ for
$i=1,\ldots,n$ then
\[
\Pr\left ( \frac{1}{n} \sum_{i=1}^n X_i -  \frac{1}{n} \sum_{i=1}^n \mathrm{E}[X_i ] \geq t \right ) \leq \exp \left (
-2 \frac{n^2 t^2}{\sum_{i=1}^n (M_i-L_i)^2 }
\right ).
\]
\end{thm}

In some scenarios, we may be given the range of a random variable {\em
  as well as} information about its variance. In those settings,
Bennett's inequality may be more appropriate
\cite{Ben62}. Classically, this inequality requires that all the
random variables have the same ceiling $M$ and the same mean $\mu$
yet can easily be generalized\footnote{Two proofs
  are provided; the first proof also requires a bottom range on the
  random variables (i.e. $-M \leq X_i \leq M$) but the second proof in
  pp. 42-3 of \cite{Ben62} does not and only requires that $X_i \leq
  M$.} as follows.

\begin{thm}[Bennett, 1962]
\label{thm:bennett}
  For a collection $X_1,\ldots,X_n$ of independent random variables
  satisfying $X_i \leq M_i$, $\mathrm{E}[X_i]=\mu_i$ and
  $\mathrm{E}[(X_i-\mu_i)^2]=\sigma_i^2$ for $i=1,\ldots,n$ and for
  any $t \geq 0$, the following holds
\[
\Pr\left ( \frac{1}{n} \sum_{i=1}^n X_i -  \frac{1}{n} \sum_{i=1}^n \mathrm{E}[X_i ] \geq t \right ) & \leq & \exp \left ( - n  \frac{v}{s^2} h \left ( \frac{ts}{v} \right ) \right )
\]
where $h(x) = (1+x) \ln(1+x) - x$, $s=\max_i M_i-\mu_i$ and
$v=\frac{1}{n} \sum_{i=1}^n \sigma_i^2$.
\end{thm}

The above theorem looks slightly more general than the classic
formula in \cite{Ben62} as well as popular variations of the original
theorem in the literature. Specifically, we allow the variables to
have heterogeneous means $\mu_i$ and heterogeneous ceilings $M_i$. In
Appendix~\ref{sct:bennettproof}, we provide a detailed derivation of
this flavor of Bennett's inequality.

Bernstein's inequality \cite{Ber24} is obtained by replacing the
function $h(x)$ with the function $g(x)=\frac{3x^2}{2x+6}$. Since
$h(x) \geq g(x)$ for all $x \geq 0$, it is known that Bennett's
inequality is strictly sharper than Bernstein's inequality.  Moreover,
Bernstein's inequality is in turn sharper than Prohorov's inequality
\cite{Pro59} and Chebyshev \cite{Tch1874} inequalities. Therefore, it
is natural to focus herein on Bennett's inequality.

Other variations of Bennett's inequality were provided by Hoeffding in
his third theorem \cite{Hoe63}. However, these variations further required
all variables to have $\mathrm{E}[X_i]=0$, a constant variance
$\sigma^2$ and a constant ceiling $M$. These variations are
therefore less useful when the random variables are heterogeneous.
Hoeffding's inequalities were
subsequently improved and extended, notably by Talagrand
\cite{Tal95,Tal95b} which, under mild conditions, gave tighter bounds
by identifying some missing factors. 

This article will propose a novel variant of Bennett's inequality. In the
homogeneous case, this variant yields a strictly tighter bound since
it avoids using an unnecessary loosening. The new inequality also
allows each individual random variable to have its own distinct and
heterogeneous range, mean and variance. These heterogeneous properties
will be exploited more subtly to obtain a bound that remains tight
even if many random variables are drawn from very different
distributions. The new bound achieves sharper rates by leveraging
Lambert's ${\cal W}$ function (specifically, the so-called principal
value of Lambert's function). As further
explicated in Appendix~\ref{sct:lambert}, this function is as
straightforward to compute as any other and enjoys a number
of analytic properties \cite{Roy2010}.

This article is organized as follows. In
Section~\ref{sct:betterbennett}, we derive a refinement of Bennett's
inequality. In Section~\ref{sct:homogeneous}, we numerically explore
the sharpness of this bound relative to other classical inequalities
in the case of homogeneous random variables which have identical
means, variances and ranges. In Section~\ref{sct:heterogeneous}, we
numerically explore the sharpness of the bound in the heterogeneous
case. In Section~\ref{sct:applications} we discuss applications of the
bound in portfolio assessment. The article then concludes with a brief
discussion.

\section{Refining Bennett's inequality}
\label{sct:betterbennett}

Given a sequence of independent random scalar variables $X_i$, the
following theorem bounds the probability that $\frac{1}{n}
\sum_{i=1}^n X_i$ will deviate above its expected value by more than
$t$.

We consider first the case where all the random variables have
homogeneous mean, ceiling and variance.

\begin{thm}
\label{thm:betterhomo} 
Let $X_1,\ldots,X_n$ be independent real-valued random variables such that $\mathrm{E}[X_i]=\mu$, $\mathrm{E}[(X_i-\mu)^2]=\sigma^2$
and $\Pr(X_i\in(-\infty,M])=1$. 
Then, for any $t\in (0,M-\mu)$, the following inequality holds
\[
\Pr\left ( \frac{1}{n} \sum_{i=1}^n X_i - \frac{1}{n} \sum_{i=1}^n \mathrm{E}[X_i] \geq t \right ) & \leq &
e^{-\lambda nt} \left (
\frac{\sigma^2}{s^2} \left ( e^{\lambda s} - 1 - \lambda s \right ) + 1 \right )^n
\]
where $s=M-\mu$ and
\[
\lambda &=& 
\frac{1}{t}+\frac{s}{\sigma^2}-\frac{1}{s}-\frac{1}{s} {\cal W}
\left ( \exp \left ( \frac{s}{t} + \frac{s^2}{\sigma^2} -1 + \ln\left (
      \frac{s-t}{t} \right) \right ) \right).
\]
\end{thm}

The above is an immediate consequence of our main theorem which
generalizes to the setting of heterogeneous random variables.

\begin{thm}
\label{thm:betterhetero} 
Let $X_1,\ldots,X_n$ be independent real-valued random variables such that $\mathrm{E}[X_i]=\mu_i$, $\mathrm{E}[(X_i-\mu_i)^2]=\sigma_i^2$
and $\Pr(X_i\in(-\infty,M_i])=1$. 
Then, for any $t\in (0,s)$, the following inequality holds
\[
\Pr\left ( \frac{1}{n} \sum_{i=1}^n X_i - \frac{1}{n} \sum_{i=1}^n \mathrm{E}[X_i] \geq t \right ) & \leq &
e^{-\lambda nt} \prod_{i=1}^n \left (
\frac{\sigma_i^2}{s_i^2} \left ( e^{\lambda s_i} - 1 - \lambda s_i \right ) + 1 \right )
\]
where
\[
\lambda &=& 
\frac{1}{{ \sum_{i=1}^n  \frac{s_i^2}{1-e^{-s_i^2/\sigma_i^2}}}}
\sum_{i=1}^n \frac{s_i^2 \lambda_i}{1-e^{-s_i^2/\sigma_i^2}}
 \\
\lambda_i &=& 
\frac{s}{t s_i}+\frac{s_i}{\sigma_i^2}-\frac{1}{s_i}-\frac{1}{s_i} {\cal W}
\left ( \exp \left ( \frac{s}{t} + \frac{s_i^2}{\sigma_i^2} -1 + \ln\left (
      \frac{s-t}{t} \right) \right ) \right) \\
s &=& \frac{1}{n} \sum_{i=1}^n s_i, \:\: s_i \:=\: M_i-\mu_i.
\]
\end{thm}
\begin{proof}
Consider the probability of interest
\[
\Pr &=& \Pr\left ( \frac{1}{n} \sum_{i=1}^n X_i - \frac{1}{n}
  \sum_{i=1}^n \mathrm{E}[X_i] \geq t \right ).
\]
Consider translated versions of the random variables $Y_i = X_i -
\mu_i$. We now have $\mathrm{E}[Y_i]=0$,
$\mathrm{E}[Y_i^2]=\sigma_i^2$ and $\Pr(Y_i \in
(-\infty,M_i-\mu_i])=1$. Applying the change of variables $Y_i = X_i -
\mu_i$ does not change the probability of interest
\[
\Pr &=& \Pr\left ( \frac{1}{n} \sum_{i=1}^n Y_i \geq t \right ) \\
 &=& \Pr\left ( e^{ \lambda \sum_{i=1}^n Y_i} \geq e^{\lambda n t} \right )
\]
where the second line holds for any $\lambda\geq 0$ by monotonic
transformation of the probability. We then apply Markov's inequality
as follows:
\[
\Pr &\leq& \inf_{\lambda \geq 0} e^{-\lambda nt} \mathrm{E} \left [ e^{\lambda \sum_{i=1}^n Y_i} \right ] \\
&=& \inf_{\lambda \geq 0} e^{-\lambda nt} 
\prod_{i=1}^n \mathrm{E} [ e^{\lambda Y_i} ].
\]
Above, the second line follows from the independence of the random
variables. Consider bounding a single term in the product, in other
words $\mathrm{E}[e^{\lambda Y_i}]$. Begin by the conjecture
(which will be subsequently proved) that the following upper bound
holds for appropriate choices of the three parameters
$(\alpha_i,\beta_i,\gamma_i)$ for all $\lambda \geq 0$:
\[
\mathrm{E}[e^{\lambda Y_i}] & \leq & \gamma_i \exp(\lambda \alpha_i) + 1 -\gamma_i - \beta_i \lambda.
\]
Clearly, when $\lambda=0$, both sides of the conjectured bound
are unity and equality is achieved.  When $\lambda=0$, the derivative
of the left hand side is 
\[
\left . \frac{\partial \mathrm{E}[e^{\lambda Y_i}]}{\partial \lambda} \right |_{\lambda=0} = \mathrm{E}[Y_i] =0.
\]
For the bound to hold locally around $\lambda=0$ both left hand side
and right hand side must have equal derivatives when they attain
equality at $\lambda=0$. This tangential contact ensures the bound
will not cross the original function as $\lambda$ varies. This forces
the following choice for the second parameter
\[
\left . \frac{\partial \left ( \gamma_i \exp(\lambda \alpha_i) + 1 -\gamma_i -  \beta_i \lambda \right )}{\partial \lambda} \right |_{\lambda=0} &=& 0 \\
\beta_i &=& \gamma_i \alpha_i.
\]
Choose $\alpha_i=M_i-\mu_i$. To satisfy the above tangential
contact constraint, it is necessary that $\beta_i = \gamma_i (M_i-\mu_i)$. The
conjectured bound is now
\[
\mathrm{E}[e^{\lambda Y_i}] & \leq & \gamma_i \exp(\lambda (M_i-\mu_i)) + 1 -\gamma_i - \gamma_i (M_i-\mu_i)\lambda.
\]
Since the above inequality makes tangential contact, taking second
derivatives of both sides with respect to $\lambda$ gives a
conservative curvature test to ensure that the bound holds. If the
right hand side has higher curvature everywhere and makes tangential
contact at $\lambda=0$ then it upper-bounds the left hand
side.  Taking second derivatives of both sides with respect to
$\lambda$ produces the curvature constraint
\[
\frac{\partial^2 \mathrm{E}[e^{\lambda Y_i}]}{\partial \lambda ^2} & \leq & 
\frac{\partial^2 \left (
 \gamma_i \exp(\lambda (M_i-\mu_i)) + 1 -\gamma_i - \gamma_i (M_i-\mu_i) \lambda
\right )
}{\partial \lambda ^2}
\\
 \mathrm{E}[e^{\lambda Y_i}Y_i^2]
&\leq & \gamma_i \exp(\lambda (M_i-\mu_i)) (M_i-\mu_i)^2.
\]
Divide both sides by $\exp(\lambda (M_i-\mu_i))$ to obtain
\[
 \mathrm{E}[e^{\lambda (Y_i-(M_i-\mu_i))}Y_i^2] &\leq & \gamma_i (M_i-\mu_i)^2.
\]
Note that $\exp(\lambda(Y_i-(M_i-\mu_i)) \leq 1$ inside the expectation since
$\lambda \geq 0$ and $Y_i \leq M_i-\mu_i$. Replacing $\exp(\lambda(Y_i-(M_i-\mu_i))$ in
the expectation with $1$ gives the following stricter condition on
$\gamma_i$ to guarantee a bound:
\[
 \mathrm{E}[e^{\lambda (Y_i-(M_i-\mu_i))}Y_i^2] \: \leq \: \mathrm{E}[Y_i^2] \: \leq \: \gamma_i (M_i-\mu_i)^2.
\]
Since $\mathrm{E}[Y_i^2]=\sigma_i^2$, the following setting for $\gamma_i$
guarantees that the curvature of the upper bound is larger than
that of the original function:
\[
\gamma_i &=& \frac{\sigma_i^2}{(M_i-\mu_i)^2}.
\]
Thus, the conjectured bound holds for any choice of $\lambda \geq 0$
and is tight at $\lambda=0$. Next define $s_i = M_i - \mu_i$ 
and rewrite the above expression as
\[
\mathrm{E}[e^{\lambda Y_i}] & \leq & 
\frac{\sigma_i^2}{s_i^2} \left ( \exp(\lambda s_i)) -1 - \lambda s_i \right ) + 1.
\]
Apply this upper bound to each individual term in the product
\[
\Pr & \leq &  \inf_{\lambda \geq 0} e^{-\lambda nt} \prod_{i=1}^n \mathrm{E}[e^{\lambda Y_i}]
\\ & \leq & \inf_{\lambda \geq 0} e^{-\lambda nt} \prod_{i=1}^n \left (
\frac{\sigma_i^2}{s_i^2} \left ( e^{\lambda s_i} - 1 - \lambda s_i \right ) + 1 \right ).
\]
This gives the bound in the theorem $\Pr \leq B(\lambda)$ where
\[
B(\lambda) &=&e^{-\lambda nt} \prod_{i=1}^n \left (
\frac{\sigma_i^2}{s_i^2} \left ( e^{\lambda s_i} - 1 - \lambda s_i \right ) + 1 \right ).
\]
What remains is to specify the choice of $\lambda$ to impute into the
formula.  We next consider ways of finding a good choice for
$\lambda$. However, we emphasize that {\em any} choice for $\lambda
\geq 0$ creates a valid upper bound on the probability $\Pr$.

We start by finding a looser upper bound on $B(\lambda)$ which we will
then minimize in order to recover $\lambda$ which we 
denote by $\lambda^*$. We start by considering $n$ arbitrary
non-negative scalar variables $t_1,\ldots,t_n$ that sum to $nt$,
i.e. $\sum_{i=1}^n t_i = nt$. We choose to set these $t_i$ as
follows
\[
t_i &=& t \frac{s_i}{\frac{1}{n} \sum_{i=1}^n s_i} \: = \: \frac{t s_i}{s}.
\]
where we have taken $s=\frac{1}{n} \sum_{i=1}^n s_i$.
Rewrite the current bound as follows
\[
B(\lambda) &=& e^{-\lambda \sum_{i=1}^n t_i} \prod_{i=1}^n \left (
\frac{\sigma_i^2}{s_i^2} \left ( e^{\lambda s_i} - 1 - \lambda s_i \right ) + 1 \right ) \\
&=& \exp \sum_{i=1}^n \left ( \log  \left (
\frac{\sigma_i^2}{s_i^2} \left ( e^{\lambda s_i} - 1 - \lambda s_i \right ) + 1 \right )-\lambda t_i \right ) \\
&=& \exp \sum_{i=1}^n b_i(\lambda)
\]
where we have defined the terms in the summation as follows
\[
b_i(\lambda) &=& \log  \left (
\frac{\sigma_i^2}{s_i^2} \left ( e^{\lambda s_i} - 1 - \lambda s_i \right ) + 1 \right ) -\lambda t_i.
\]
The minimizer of $b_i(\lambda)$ is obtained in closed-form via the
Lambert ${\cal W}$ function as follows
\[
\lambda_i^* &=& \arg \min_\lambda b_i(\lambda) \\
&=& \frac{1}{t_i}+\frac{s_i}{\sigma_i^2}-\frac{1}{s_i}-\frac{1}{s_i} {\cal W}
\left ( \exp \left ( \frac{s_i}{t_i} + \frac{s_i^2}{\sigma_i^2} -1 + \ln\left (
      \frac{s_i-t_i}{t_i} \right) \right ) \right) \\
&=& \frac{s}{ts_i}+\frac{s_i}{\sigma_i^2}-\frac{1}{s_i}-\frac{1}{s_i} {\cal W}
\left ( \exp \left ( \frac{s}{t} + \frac{s_i^2}{\sigma_i^2} -1 + \ln\left (
      \frac{s-t}{t} \right) \right ) \right)
\]
where in the last line we have inserted the choice for $t_i$. Note,
that Theorem~\ref{thm:nonnegative} in the Appendix ensures that the
minimizers are non-negative, in other words, $\lambda_i^*\geq 0.$ 

Next, derive the curvature of $b_i(\lambda)$ and upper-bound it as follows
\[
\frac{\partial^2 b_i(\lambda)}{\partial \lambda^2}
&=& \frac{\sigma_i^2 e^{\lambda s_i}}{\frac{\sigma_i^2}{s_i^2} (e^{\lambda s_i} - \lambda s_i - 1) + 1} - \left ( 
\frac{ \frac{\sigma_i^2}{s_i} (e^{\lambda s_i}-1)}{\frac{\sigma_i^2}{s_i^2} (e^{\lambda s_i} - \lambda s_i - 1) + 1}
\right )^2 \\
& \leq &  \frac{\sigma_i^2 e^{\lambda s_i}}{\frac{\sigma_i^2}{s_i^2} (e^{\lambda s_i} - \lambda s_i - 1) + 1}.
\]
Above, we bounded the curvature simply by dropping the last negative
term.  Taking derivatives and setting to zero gives the maximum of the
right hand side above which yields $\lambda=s_i/\sigma_i^2$. Inserting
this value of $\lambda$ into the bound gives a supremum on the
curvature
\[
\frac{\partial^2 b_i(\lambda)}{\partial \lambda^2} &\leq & 
 \frac{\sigma_i^2 \exp(s_i^2/\sigma_i^2)}{\frac{\sigma_i^2}{s_i^2} (e^{s_i^2/\sigma_i^2} - s_i^2/\sigma_i^2 - 1) + 1} \\
&=& \frac{s_i^2}{1-e^{-s_i^2/\sigma_i^2}}.
\]

We now form a quadratic upper bound for each $b_i(\lambda)$ term as follows,
\[
b_i(\lambda) & \leq &  \frac{s_i^2}{1-e^{-s_i^2/\sigma_i^2}}
\frac{(\lambda-\lambda_i^*)^2}{2} + b_i(\lambda_i^*).
\]
The above holds since both left hand side and right hand side are equal when
$\lambda=\lambda_i^*$. Furthermore, the gradients of both the left hand
side and the right hand side are zero when $\lambda=\lambda_i^*$. Finally,
the curvature of the left hand side is always less than the curvature
of the right hand side. Therefore, the quadratic on the right hand
side must be an upper bound on $b_i(\lambda)$. Replacing each
$b_i(\lambda)$ term with its corresponding quadratic bound gives an overall upper
bound on $B(\lambda)$ as follows
\[
B(\lambda) &=& \exp\left ( \sum_{i=1}^n b_i(\lambda) \right ) \\
& \leq & \exp \left ( \sum_{i=1}^n \frac{s_i^2}{1-e^{-s_i^2/\sigma_i^2}}
 \frac{(\lambda-\lambda_i^*)^2}{2} + b_i(\lambda_i^*) \right ).
\]
It is easy to minimize the right hand side analytically over $\lambda$ to obtain
\[
\lambda^* &=& \frac{1}{{ \sum_{i=1}^n  \frac{s_i^2}{1-e^{-s_i^2/\sigma_i^2}}}}
\sum_{i=1}^n \frac{s_i^2 \lambda_i^*}{1-e^{-s_i^2/\sigma_i^2}}
\]
which yields the theorem.
\end{proof}

An interesting property of Theorem~\ref{thm:betterhetero} is that it
carefully incorporates heterogeneous information about the different
random variables. Rather than simply averaging variances (as in
Bennett's inequality), we compute more complicated interactions
between the variances $\sigma_i^2$ and the spreads
$s_i=M_i-\mu_i$. This subtle combination of information about the
heterogeneous random variables will yield significant improvements
over Bennett's inequality. Furthermore, in the homogeneous setting
(which emerges if all $\sigma_i=\sigma$ and all $s_i=s$), it is clear
that the new bound in Theorem~\ref{thm:betterhomo} is circumventing
loosening steps in Bennett's inequality. Therefore, our proposed
bounds are strictly tighter than Bennett's.

We can also consider a counterpart of Theorem~\ref{thm:betterhetero}
which uses information about the bottom range on $X_i$, namely $L_i
\leq X_i$, instead of the ceiling. Note that it is also
straightforward to derive counterparts of Bennett's and Bernstein's
bounds in such a setting as well.

\begin{thm}
\label{thm:betterreversed} 
Let $X_1,\ldots,X_n$ be independent real-valued random variables such that $\mathrm{E}[X_i]=\mu_i$, $\mathrm{E}[(X_i-\mu_i)^2]=\sigma_i^2$
and $\Pr(X_i\in [L_i,\infty))=1$. 
Then, for any $t\in (0,s)$, the following inequality holds
\[
\Pr\left ( \frac{1}{n} \sum_{i=1}^n X_i - \frac{1}{n} \sum_{i=1}^n \mathrm{E}[X_i] \leq -t \right ) & \leq &
e^{-\lambda nt} \prod_{i=1}^n \left (
\frac{\sigma_i^2}{s_i^2} \left ( e^{\lambda s_i} - 1 - \lambda s_i \right ) + 1 \right )
\]
where
\[
\lambda &=& 
\frac{1}{{ \sum_{i=1}^n  \frac{s_i^2}{1-e^{-s_i^2/\sigma_i^2}}}}
\sum_{i=1}^n \frac{s_i^2 \lambda_i}{1-e^{-s_i^2/\sigma_i^2}}
 \\
\lambda_i &=& 
\frac{s}{t s_i}+\frac{s_i}{\sigma_i^2}-\frac{1}{s_i}-\frac{1}{s_i} {\cal W}
\left ( \exp \left ( \frac{s}{t} + \frac{s_i^2}{\sigma_i^2} -1 + \ln\left (
      \frac{s-t}{t} \right) \right ) \right) \\
s &=& \frac{1}{n} \sum_{i=1}^n s_i, \:\: s_i \:=\: \mu_i-L_i.
\]
\end{thm}
\begin{proof}
  The proof is entirely analogous to the one for
  Theorem~\ref{thm:betterhetero}.
\end{proof}

\section{Experiments with homogeneous random variables}
\label{sct:homogeneous}

This section numerically compares the new bound with Hoeffding's,
Bennett's and Bernstein's bounds which are well-known classical
concentration inequalities. This section will only
consider the homogeneous setting, $M_i=M$, $\mu_i=\mu$ and
$\sigma_i=\sigma$ for all $i=1,\ldots,n$ random variables.

Specifically, we will compare our bound in
Theorem~\ref{thm:betterhomo} against Bennett's inequality in
Theorem~\ref{thm:bennett} which uses $h(x)=(1+x)\log(1+x)-x$.
Similarly, we consider Bernstein's inequality \cite{BouBouLug04} which
is identical to Bennett's yet uses $g(x)=\frac{3x^2}{2x+6}$ in the
place of $h(x)$. Finally, we apply Hoeffding's inequality as in
Theorem~\ref{thm:hoeffding} whose rate has no dependence on the
variance of the random variables.

To directly compare these bounds, we will explore various values of
$M$, $\mu$ and $\sigma$ to see how they compare against each other.
For practical visualization purposes, we first note that all bounds
scale with $n$ in the same manner. Therefore, without loss of
generality, we set $n=1$ throughout our experiments. Furthermore,
since we can scale $(M,\mu,\sigma)$ by an arbitrary factor without
changing the bounds, we simply lock $M=1$ to remove this source of
redundancy in our experimental exploration. Figure~\ref{fig:bounds1}
and Figure~\ref{fig:bounds2}, depict the convergence rates for the
four different concentration inequalities under various choices of
$\mu \in \{-\frac{1}{2},0,\frac{1}{2}\}$ and $\sigma \in
\{1,1/2,1/4,1/8\}$.  Since Hoeffding's bound also requires a value for
the bottom range of the random variable, $L$, we make a simple
arbitrary choice to set it to $L=-M$. In fact, many applications of
Bennett and Bernstein typically assume that $|X_i|\leq M$ (though it
is not strictly necessary to do so). Also, note that setting $L=-M$
does not violate the elementary inequality $\sigma_i \leq (M_i-L_i)/2$
in any of our experiments. By observing the log-probability for each
bound as $t$ varies, it is possible to see which inequalities are
tighter (i.e. yielding an exponentially smaller deviation
probability). Bounds with lower (negative) rates indicate faster
convergence of the average to its expected value.

Clearly, the new bound is strictly sharper than Bernstein's and
Bennett's inequalities in the homogeneous setting. In fact, we know
that this must be true since Bennett makes an unnecessary loosening
step and Bernstein follows it with yet another unnecessary loosening
step. Hoeffding's bound performs poorly unless the variance $\sigma^2$
is large, i.e. it is close to the maximum value it can have while
still respecting the elementary inequality $\sigma_i \leq
(M_i-L_i)/2$. At that setting, the variance is essentially providing
very little useful information and so the variance-based inequalities
(such as Bennett's, Bernstein's and the proposed bound) are no longer
relevant. Otherwise, the new bound clearly dominates the classical
inequalities in our experiments.

It is important to note that these rate quantities will be multiplied
by the number of observations $n$ and then exponentiated to obtain
bounds on the probability. Therefore, the advantages of the bounds
relative to each other will be drastically magnified as $n$ grows.

\section{Experiments with heterogeneous random variables}
\label{sct:heterogeneous}

In Section~\ref{sct:homogeneous} we compared the new bound to
classical concentration inequalities when all the random variables are
homogeneous. We here consider bounding the probability
$\Pr(\frac{1}{n} \sum_{i=1}^n X_i - \frac{1}{n} \sum_{i=1}  \mathrm{E}[X_i ] \geq
t)$ when we deal with independent random variables $X_i$ that are {\em
  not} identically distributed but rather have their own distinct
heterogeneous values of $M_i$, $\mu_i$ and $\sigma_i$. In these
experiments, the advantages of the bound can sometimes be
dramatic.

We will consider various random choices of $M_i$, $\mu_i$, $\sigma_i$
and $t$. These synthetic experiments allow us to compare our new
inequality relative to the Bernstein, Bennett and Hoeffding
inequalities in the heterogeneous setting. To generate synthetic
problems, we set $M_i$ to be the absolute value of random draws from a
white Gaussian (e.g. with zero-mean and unit variance). We set $L_i$
to be the negated absolute values drawn from a white Gaussian. We set
$\mu_i$ equal to a uniform value drawn in the interval $[L_i,M_i]$. We
then choose $\sigma_i$ uniformly from $[0,\frac{1}{2z}(M_i-L_i)]$ for
various choices of $z\geq 1$. This way, we explore different levels of
variance without ever violating the elementary inequality $\sigma_i
\leq (M_i-L_i)/2$. Finally, we set $t$ by sampling a scalar from the
uniform distribution and multiplying it by the value $s=\frac{1}{n}
\sum_{i=1}^n M_i-\mu_i$ which was introduced in our bound.

We compute the bound using Theorem~\ref{thm:betterhetero} in the
heterogeneous setting. To compute Bennett's inequality in the
heterogeneous setting, we use the formula in
Theorem~\ref{thm:bennett}. Similarly, by replacing the $h(x)$ function
with the $g(x)$ function, we compute Bernstein's inequality. All three
of these approaches ignore information in the $L_i$ values. To compute
a bound using Hoeffding's inequality, we apply
Theorem~\ref{thm:hoeffding} which ignores information about the
$\sigma_i$ and $\mu_i$ values.

In Figure~\ref{fig:scatter1} and Figure~\ref{fig:scatter2}, we see the
log-probabilities for the new bound on the x-axis and the
log-probabilities for the classical inequalities on the
y-axis. Several experiments are shown for $n=1,10,10$ and for
$z=1,2,10,100$. Whenever a coordinate marker is above the diagonal
line, the new bound is performing better for that particular random
experiment. Clearly, the new bound is outperforming Bennett's,
Bernstein's and Hoeffding's bounds.  When $n=1$ in the top of the
figures, we are back to the homogeneous case where the bound must
strictly outperform Bennett's and Bernstein's inequalities. It also
seems to frequently outperform Hoeffding's. As we increase $n$, the
advantages of the bound become even more dramatic in the heterogeneous
case.  When $z$ is small, the variances $\sigma_i^2$ are large and
potentially close to their maximum allowable values (e.g. prior to
violating the elementary inequality $\sigma_i \leq
(M_i-L_i)/2$). Therefore, variance does not provide much information
about the distribution of the random variables. Meanwhile, when $z$ is
large, the the variance values are smaller and Hoeffding's bound
becomes extremely loose since it ignores variance information.  The
new bound seems to frequently outperform the classical inequalities.

\section{An application in portfolio optimization}
\label{sct:applications}

There are many natural applications of Theorems~\ref{thm:betterhomo},
~\ref{thm:betterhetero}, and ~\ref{thm:betterreversed}.  As a 
motivating example, consider a financial portfolio with several
independent investments. Each investment $i$ will provide a payoff
$X_i$ from an unknown distribution. We may know a priori the minimum
payoff $L_i$, the expected payoff $\mu_i$ and the variance of the
payoff $\sigma_i^2$ for investment $i$. We are interested in the
probability that the sum total of our investments will under-perform
its expected value by $t$. For example, we may want to compute the
probability that the portfolio will {\em under-perform} and produce a
total payoff that is smaller than some risk-free payoff $\tau$. The
quantity of interest is $\Pr(\sum_{i=1}^n X_i \leq \tau)$.

Using Theorem~\ref{thm:betterreversed}, it is straightforward
to upper-bound the probability that a portfolio under-performs.  We
are given information about each $X_i$ such as its $L_i$, $\mu_i$ and
$\sigma_i$ and we are given a threshold $\tau$ for the portfolio to be
worthwhile. If the total payoff from the investment falls below
$\tau$, it has {\em under-performed}. Our upper bound holds without
making any further parametric assumptions about the distribution of
the payoffs $X_1,\ldots,X_n$. Conversely, many practitioners make
Gaussian assumptions or parametric assumptions about portfolios and
payoff distributions \cite{Mar52}. A non-parametric approach
remains agnostic and may be better matched to real-world
settings.

Consider the following toy example. We have $n=2$
investments. Investment 1 has an expected payoff of $\mu_1 = \$30$ with
a standard deviation of $\sigma_1=\$25$. Investment 2 has an expected
payoff of $\mu_2 = \$100$ with a standard deviation of
$\sigma_2=\$20$. Investment 1 has a floor on its payoff of
$L_1=\$25$. Meanwhile, Investment 2 can potentially yield as little as 
$L_2=\$5$ in terms of payout. For the portfolio to be worthwhile, we are
told that the total payoff of both investments must be at least $\$74$
(or the average payoff across both investments must be at least
$\$37$). Otherwise, the portfolio is under-performing.

According to our new bound in Theorem~\ref{thm:betterreversed},
the probability of under-performing is less than 39.1\%.  We next
apply Bennett's inequality and Bernstein's inequality to this
portfolio problem. Consider Bennett's inequality in
Theorem~\ref{thm:nicerbennett}. Just as we were able to reverse the
bound in Theorem~\ref{thm:betterhetero} to obtain
Theorem~\ref{thm:betterreversed}, it is possible to obtain
reversed versions of Bennett's and Bernstein's inequalities. Bennett's
says the probability is at most 50.1\% and Bernstein says it is at most
57.2\%.

To apply Hoeffding's inequality, we also need an upper bound on the
investment's payouts (e.g. $M_1$ and $M_2$). Recall the elementary
formula $\sigma_i \leq (M_i-L_i)/2$. This gives us the {\em
  most-optimistic} value for ${\hat M}_i=\max(L_i+2\sigma_i, \mu_i)$.
This setting helps tighten the Hoeffding bound as much as
possible. Technically, the {\em most-optimistic} setting of the
Hoeffding bound can be quite erroneous and misleading. In fact, if we
make additional assumptions about the problem, then Hoeffding is not
actually computing a bound. We are making assumptions that may not be
true about the original problem. Nevertheless, we use the
heterogeneous $L_i$ and imputed ${\hat M}_i$ in
Theorem~\ref{thm:hoeffding}. Hoeffding says the probability of
under-performing is less than 58.1\%. Surprisingly, this is still
worse than the (valid) estimate using our novel bound.

The new bound gives the best estimate and shows that the payoff on our
investments is more likely to meet the target total payoff of $\$74$
(or an {\em average} payoff of $\$37$ across both investments). The
bound on the probability of 39.1\% shows that the investment portfolio
is worth the risk.

Rather than simply {\em bound} the probability a given portfolio
under-performs, we may wish to {\em find an optimal portfolio}.
Consider $n$ possible investments where we aim to
optimally allocate funds by computing the proportion $\alpha_i\geq 0$ of our budget that is allocated to
investment $i$ for $i=1,\ldots n$. Clearly, budget proportions
sum to unity and therefore $\sum_{i=1}^n \alpha_i=1$. Let $X_i$ represent a random
variable which equals the return of the $i$'th investment. Assume we
know a priori that $X_i$ has mean $\mu_i$, deviation $\sigma_i$ and floor
$L_i$. We wish to find $\alpha_1,\ldots,\alpha_n$ that minimize the
probability $\Phi$ that the portfolio generates less than some
targeted return $\tau$. We can compute $\Phi$ as
follows
\[
\Phi &=& \Pr\left ( \sum_{i=1}^n \alpha_i X_i \leq \tau \right ) \:\:
= \:\:  \Pr \left (
\frac{1}{n} \sum_{i=1}^n {\tilde X}_i - \frac{1}{n} \sum_{i=1}^n
  {\tilde \mu}_i \leq -\frac{1}{n} t \right ).
\]
On the right, we have merely rewritten the probability after the change of variables ${\tilde
  X}_i=\alpha_i X_i$ and $t = \sum_{i=1}^n {\tilde \mu}_i -
\tau$. The new random variables ${\tilde X}_i$ have mean
${\tilde \mu}_i = \alpha_i \mu_i$, variance ${\tilde
  \sigma_i}^2=\alpha_i^2 \sigma_i^2$ and floor ${\tilde L}_i=\alpha_i
L_i$. Apply Theorem~\ref{thm:betterreversed} which holds for any
$\lambda \geq 0$ to obtain
\[
\Phi &\leq& e^{-\lambda t} \prod_{i=1}^n \left (
\frac{\alpha_i^2 \sigma_i^2}{\alpha_i^2 (\mu_i - L_i)^2} \left (
  e^{\lambda \alpha_i (\mu_i - L_i)} - 1 - \lambda \alpha_i(\mu_i -
  L_i) \right ) + 1 \right ) \\
&=& e^{-\lambda t} \prod_{i=1}^n \left (
\frac{\sigma_i^2}{(\mu_i - L_i)^2} \left (
  e^{\lambda_i (\mu_i - L_i)} - 1 - \lambda_i(\mu_i -
  L_i) \right ) + 1 \right )
\]
where in the second line we have simply defined $\lambda_i=\alpha_i \lambda$. Insert the definition of $t$ into the above
\[
\Phi
& \leq & e^{-\lambda (\sum_{i=1}^n \alpha_i
  (\mu_i - \tau))} \prod_{i=1}^n \left (
\frac{\sigma_i^2}{(\mu_i - L_i)^2} \left (
  e^{\lambda_i (\mu_i - L_i)} - 1 - \lambda_i(\mu_i -
  L_i) \right ) + 1 \right ).
\]
Recall that $\lambda_i = \lambda \alpha_i$ and therefore
\[
\Phi & \leq & \prod_{i=1}^n e^{- \lambda_i 
  (\mu_i - \tau)} \left (
\frac{\sigma_i^2}{(\mu_i - L_i)^2} \left (
  e^{\lambda_i (\mu_i - L_i)} - 1 - \lambda_i(\mu_i -
  L_i) \right ) + 1 \right ). 
\]
We wish to minimize the bound on the right hand side over $\lambda
\geq 0$ as
well as to minimize it over $\alpha_i \geq 0$ for $i=1,\ldots,n$ subject to $\sum_{i=1}^n
\alpha_i=1$. An equivalent problem is to minimize the right hand side
of the inequality over $\lambda_i \geq 0$ for $i=1,\ldots,n$. The solution can be found by independently minimizing each term in the
product $\prod_{i=1}^n$ above over the $\lambda_i$ that appears in
it. The solution for $\lambda_i$ is a
straightforward reapplication of the derivations in the proof of
Theorem~\ref{thm:betterhetero},
\[
\lambda_i&=&
\frac{1}{\mu_i-\tau}+\frac{\mu_i-L_i}{\sigma_i^2}-\frac{1}{\mu_i-L_i}\left
  ( 1 + {\cal W}
\left ( \left (
      \frac{\tau-L_i}{\mu_i-\tau} \right) e^{
      \frac{\mu_i-L_i}{\mu_i-\tau} + \frac{(\mu_i-L_i)^2}{\sigma_i^2}
      -1  } \right) \right )
\]
and shows that we must require $\tau \in [\mu_i,L_i]$ to obtain valid
numerical solutions.

To recover the optimal proportions for our budget, we merely
compute $\alpha_i=\lambda_i / \sum_{i=1}^n \lambda_i$. Given a $\tau$,
to recover the
optimized upper bound on $\Phi$, insert the suggested values of
$\lambda_i$ into the final bound above. Alternatively, rather than
specifying a $\tau$ value, a user may prefer to specify how much
deviation from the expected return he is willing to tolerate by
selecting $t$. In that case, to avoid numerical problems, it is
straightforward to see that $t \in (0,\min_i (\mu_i-L_i))$.

Figure~\ref{fig:portfolio2} depicts the upper bound on the probability
we will obtain a return less than $\tau$ for 3 investments having a
floor of 0 and $\mu_1=0.3030, \mu_2=0.2400, \mu_3=0.6178$ with
$\sigma_1=0.2601,\sigma_2=0.5248,\sigma_3=0.7645$ (top panel). In the bottom panel, the optimal
portfolio distribution is depicted across $\tau$ values. 
Figure~\ref{fig:portfolio1} depicts the upper bound on the probability
we will obtain a return less than $\tau$ for 4 investments having a
floor of 0 and $\mu_1=0.1474, \mu_2=0.6088, \mu_3=0.1785, \mu_4=0.7585$ with
$\sigma_1=0.0593,\sigma_2=0.6218,\sigma_3=0.2183,\sigma_4=0.4597$ (top panel). In the bottom panel, the optimal
portfolio distribution is depicted across $\tau$ values.

\section{Conclusions}

A new bound was proposed that characterizes the convergence of the
average of independent bounded random variables towards its expected
value. In the homogeneous case, the new bound is strictly sharper than
Bennett's and Bernstein's inequalities and very often outperforms
Hoeffding's inequality.  The bound also readily applies in settings
where the random variables are not identically-distributed and may
have heterogeneous values for their range, expected value, and
variance. In the heterogeneous case, the new bound sometimes
dramatically outperforms the classical inequalities as well. The bound
appears useful in portfolio optimization as well as potentially other
application areas. Deriving the bound involved the use of a peculiar
transcendental function known as Lambert's ${\cal W}$ function.  While
Lambert's ${\cal W}$-function has been known since the 1779 paper by
Leonhard Euler it has only been popularized in the 1980's. It may be
helpful in the development of other concentration inequalities.

\appendix

\section{Lambert's {\cal W}}
\label{sct:lambert}

The Lambert ${\cal W}$ function is defined
by the solution of the equation
\[
{\cal W} &=& \ln(x) - \ln({\cal W}).
\]
Although this function is now integrated in most modern
mathematical software, an iterative scheme is provided below to show
how it is numerically recovered. The proposed numerical scheme
estimates ${\cal W}(x)$ by first guessing a solution ${\cal W}$ and
then using the following iterations until convergence:
\[
{\cal W} \leftarrow {\cal W} - \frac{ {\cal W} \exp({\cal W})-x} { ({\cal W}+1)\exp({\cal W})- ({\cal W}+2) \frac{{\cal W}\exp({\cal W})-x}{2 {\cal W}+2} }.
\]
An alternative is to manipulate Lambert's ${\cal W}$ function
after the exponentiation operation which yields better numerical
results and solves the equation
\[
{\cal W}(\exp(x)) &=& x - \ln({\cal W}).
\]
If $x \geq 0$, initialize with ${\cal W}=x$ and iterate
the following rule to recover ${\cal W}(\exp(x))$:
\[
{\cal W} & \leftarrow & {\cal W} - \frac{{\cal W}-\exp(x-{\cal W})}{{\cal W}+1-({\cal W}+2) \frac{{\cal W}-\exp(x-{\cal W})}{2{\cal W}+2} }.
\]
If $x \leq 0$, initialize with ${\cal W}=1$ and
iterate the following rule to recover ${\cal W}(\exp(x))$:
\[
{\cal W} \: \leftarrow \: {\cal W} - \frac{{\cal W}\exp({\cal W}-x)-1}{({\cal W}+1)\exp({\cal W}-x)-({\cal W}+2)\frac{{\cal W}\exp({\cal W}-x)-1}{2 {\cal W}+2}}
.
\]
These updates circumvent numerical problems that may arise during
evaluation of the bound in Theorems~\ref{thm:betterhomo}, ~\ref{thm:betterhetero}, and ~\ref{thm:betterreversed}.

\section{Closed-form minimizer}

The following is useful to prove Theorem~\ref{thm:betterhetero}.

\begin{thm}
The minimizer of $\log( q (e^{\lambda r } - {\lambda r} - 1) + 1) - \lambda t$ is non-negative
for $q \geq 0, r \geq 0$, and $t\in(0,r)$.
\label{thm:nonnegative}
\end{thm}
\begin{proof}
The minimizer is given by the closed-form expression 
\[
\lambda^* &=& \frac{1}{t}+\frac{1}{qr}-\frac{1}{r}-\frac{1}{r} {\cal W}
\left ( \exp \left ( \frac{r}{t} + \frac{1}{q} -1 + \ln\left (
      \frac{r-t}{t} \right) \right ) \right).
\]
Non-negativity therefore implies
\[
\frac{r}{t} + \frac{1}{q} - 1 & \geq & {\cal W} \left ( \exp \left (
    \frac{r}{t} + \frac{1}{q} -1 + \ln\left ( \frac{r-t}{t} \right)
  \right ) \right).
\]
Consider the change of variables $x=\frac{r}{t}+\frac{1}{q}-1$ and $y=\frac{1}{q}$.  Inserting these into the above required inequality yields
\begin{eqnarray}
x & \geq & {\cal W}( (x-y) e^{x} ).
\label{eqn:xandy}
\end{eqnarray}
Since, by definition $t \in (0,r)$, it is clear that $x
\geq y \geq 0$. Equation~\ref{eqn:xandy} is certainly true if the
maximization over allowable $y$ on the right hand side (a stricter
requirement) holds. In other words, if the following is true
\[
x & \geq & \max_{y \: {\rm s.t.} \: x \geq y \geq 0} {\cal W}( (x-y) e^{x} ),
\]
then Equation~\ref{eqn:xandy} must hold.  The maximum of the right
hand side over $y$ is attained when $y=0$ since ${\cal W}$ is a
monotonically increasing function in its argument. Thus, setting $y=0$
yields
\[
x & \geq & {\cal W}(xe^x ).
\]
By definition, $\vartheta = {\cal W}(\vartheta) \exp({\cal W}(\vartheta))$. Therefore, ${\cal W}(x e^x) = x$ (as long as $x \geq 0$ which has been established). Therefore, the above stricter requirement is equivalent to
\[
x & \geq & x
\]
which clearly holds. This demonstrates the optimal solution for
$\lambda$ remains non-negative as is required for a valid application
of Markov's inequality.
\end{proof}

\section{Bennett's inequality}
\label{sct:bennettproof}

Since so many variations of Bennett's inequality have been used in the
literature, we rederive it below in the generalized setting.

\begin{thm}
\label{thm:nicerbennett}
  For a collection $X_1,\ldots,X_n$ of independent random variables
  satisfying $X_i \leq M_i$, $\mathrm{E}[X_i]=\mu_i$ and
  $\mathrm{E}[(X_i-\mu_i)^2]=\sigma_i^2$ for $i=1,\ldots,n$ and for
  any $t \geq 0$, the following holds
\[
\Pr\left ( \frac{1}{n} \sum_{i=1}^n X_i -  \frac{1}{n} \sum_{i=1}^n \mathrm{E}[X_i ] \geq t \right ) & \leq & \exp \left ( - n  \frac{v}{s^2} h \left ( \frac{ts}{v} \right ) \right )
\]
where $h(x) = (1+x) \ln(1+x) - x$, $s=\max_i M_i-\mu_i$ and
$v=\frac{1}{n} \sum_{i=1}^n \sigma_i^2$.
\end{thm}
\begin{proof}
Consider the probability of interest
\[
\Pr &=& \Pr\left ( \frac{1}{n} \sum_{i=1}^n X_i - \frac{1}{n}
  \sum_{i=1}^n \mathrm{E}[X_i] \geq t \right ).
\]
Consider translated versions of the random variables $Y_i = X_i -
\mu_i$. We now have $\mathrm{E}[Y_i]=0$,
$\mathrm{E}[Y_i^2]=\sigma_i^2$ and $\Pr(Y_i \in
(-\infty,M_i-\mu_i])=1$. Applying the change of variables $Y_i = X_i -
\mu_i$ does not change the probability of interest
\[
\Pr &=& \Pr\left ( \frac{1}{n} \sum_{i=1}^n Y_i \geq t \right ) \\
 &=& \Pr\left ( e^{ \lambda \sum_{i=1}^n Y_i} \geq e^{\lambda n t} \right )
\]
where the second line holds for any $\lambda\geq 0$ by monotonic
transformation of the probability. We then apply Markov's inequality
as follows:
\[
\Pr &\leq& \inf_{\lambda \geq 0} e^{-\lambda nt} \mathrm{E} \left [ e^{\lambda \sum_{i=1}^n Y_i} \right ] \\
&=& \inf_{\lambda \geq 0} e^{-\lambda nt} 
\prod_{i=1}^n \mathrm{E} [ e^{\lambda Y_i} ].
\]
Above, the second line follows from the independence of the random
variables. Consider bounding a single term in the product, in other
words $\mathrm{E}[e^{\lambda Y_i}]$.
Begin by the
conjecture (which will be subsequently proved) that the following
upper bound holds for appropriate choices of the parameters
$(\alpha_i,\beta_i,\gamma_i)$ for all $\lambda \geq 0$:
\[
\mathrm{E}[e^{\lambda Y_i}] & \leq & \gamma_i \exp(\lambda \alpha_i) + 1 -\gamma_i - \beta_i \lambda.
\]
Clearly, when $\lambda=0$, both sides of the conjectured bound
are unity and equality is achieved.  When $\lambda=0$, the derivative
of the left hand side is also zero since
\[
\left . \frac{\partial \mathrm{E}[e^{\lambda Y_i}]}{\partial \lambda} \right |_{\lambda=0} = \mathrm{E}[Y_i] = 0.
\]
For the bound to hold locally around $\lambda=0$ both left hand side
and right hand side must have equal derivatives when they attain
equality at $\lambda=0$. This tangential contact ensures the bound
will not cross the original function as $\lambda$ varies. This forces
the following choice for the second parameter of the bound
\[
\left . \frac{\partial \left ( \gamma_i \exp(\lambda \alpha_i) + 1 -\gamma_i -  \beta_i \lambda \right )}{\partial \lambda} \right |_{\lambda=0} &=& 0 \\
\beta_i &=& \gamma_i \alpha_i.
\]
Let $s=\max_i M_i-\mu_i$. Choose 
$\alpha_i=s$ and, therefore, to satisfy the above tangential
contact constraint, it is necessary that $\beta_i = \gamma_i s$. The
conjectured bound is now
\[
\mathrm{E}[e^{\lambda Y_i}] & \leq & \gamma_i \exp(\lambda s) + 1 -\gamma_i - \gamma_i s \lambda.
\]
Since the above inequality makes tangential contact, taking second
derivatives of both sides with respect to $\lambda$ gives a
conservative curvature test to ensure that the bound holds. If the
right hand side has higher curvature everywhere and makes tangential
contact at $\lambda=0$ then it upper-bounds the left hand
side.  Taking second derivatives of both sides with respect to
$\lambda$ produces the curvature constraint
\[
\frac{\partial^2 \mathrm{E}[e^{\lambda Y_i}]}{\partial \lambda ^2} & \leq & 
\frac{\partial^2 \left (
 \gamma_i \exp(\lambda s) + 1 -\gamma_i - \gamma_i s \lambda
\right )
}{\partial \lambda ^2}
\\
 \mathrm{E}[e^{\lambda Y_i}(Y_i)^2]
&\leq & \gamma_i \exp(\lambda s) s^2.
\]
Divide both sides by $\exp(\lambda s)$ to obtain
\[
 \mathrm{E}[e^{\lambda (Y_i-s)}(Y_i)^2] &\leq & \gamma_i s^2.
\]
Note that $\exp(\lambda(Y_i-s)) \leq 1$ since $Y_i \leq M_i-\mu_i$,
$s=\max_i M_i -\mu_i$ and $\lambda \geq 0$. Replacing
$\exp(\lambda(Y_i-s))$ in the expectation with $1$ gives the following
stricter condition on $\gamma_i$ to guarantee a bound:
\[
\mathrm{E}[Y_i^2] &\leq & \gamma_i s^2.
\]
Since $\mathrm{E}[Y_i^2]=\sigma_i^2$, the following setting for $\gamma_i$
guarantees that the curvature of the upper bound is larger than
that of the original function:
\[
\gamma_i &=& \frac{\sigma_i^2}{s^2}.
\]
Thus, with the parameters
$(\alpha_i=s,\beta_i=\frac{\sigma_i^2}{s},\gamma_i=\frac{\sigma_i^2}{s^2})$,
the conjectured bound holds for any choice of $\lambda \geq 0$ and is
tight at $\lambda=0$.  This upper bound applies to each term
as follows
\[
\mathrm{E}[e^{\lambda Y_i}] & \leq & 
\frac{\sigma_i^2}{s^2} \exp(\lambda s) + 1 - \frac{\sigma_i^2}{s^2} - \frac{\lambda \sigma_i^2}{s}.
\]
Rewrite the above expression as
\[
\mathrm{E}[e^{\lambda Y_i}] & \leq & 
\frac{\sigma_i^2}{s^2} \left ( e^{\lambda s} - 1
  - \lambda s \right ) + 1.
\]
Apply this upper bound to each individual term in the product
as follows
\[
\Pr & \leq &  \inf_{\lambda \geq 0} e^{-\lambda nt} \prod_{i=1}^n \mathrm{E}[e^{\lambda Y_i}]
\\ & \leq & \inf_{\lambda \geq 0} e^{-\lambda nt} \prod_{i=1}^n \left (
\frac{\sigma_i^2}{s^2} \left ( e^{\lambda s} - 1 - \lambda s \right ) + 1 \right ).
\]
Bennett then invokes the inequality
$1+x \leq \exp(x)$ as follows
\[
\Pr
 & \leq & \inf_{\lambda \geq 0} e^{-\lambda nt} \prod_{i=1}^n \left (
\frac{\sigma_i^2}{s^2} \left ( e^{\lambda s} - 1 - \lambda s \right ) + 1 \right ) \\
& \leq &  \inf_{\lambda \geq 0} e^{-\lambda nt} \prod_{i=1}^n \exp \left( \frac{\sigma_i^2}{s^2} \left ( e^{\lambda s} - 1 - \lambda s \right )  \right ).
\]
This inequality conveniently groups the $\exp()$ functions together
across the product $\prod_{i=1}^n$ to yield
\[
\Pr & \leq &  \inf_{\lambda \geq 0} \exp \left( \frac{\sum_{i=1}^n \sigma_i^2}{s^2} \left ( e^{\lambda s} - 1 - \lambda s \right )  - \lambda n t\right ).
\]
Minimizing the right hand side over $\lambda$ and re-inserting gives the bound
\[
\Pr & \leq &
\exp \left ( - n  \frac{v}{s^2} h \left ( \frac{ts}{v} \right ) \right )
\]
where $h(x) = (1+x) \ln(1+x) - x$ and $v=\frac{1}{n} \sum_{i=1}^n \sigma_i^2.$
\end{proof}

\bibliography{all}
\bibliographystyle{plain}

\begin{figure}[htbp]
  \center
\setlength\tabcolsep{2pt}
  \begin{tabular}[b]{cc}
    \epsfxsize=2.46in
    \epsfbox{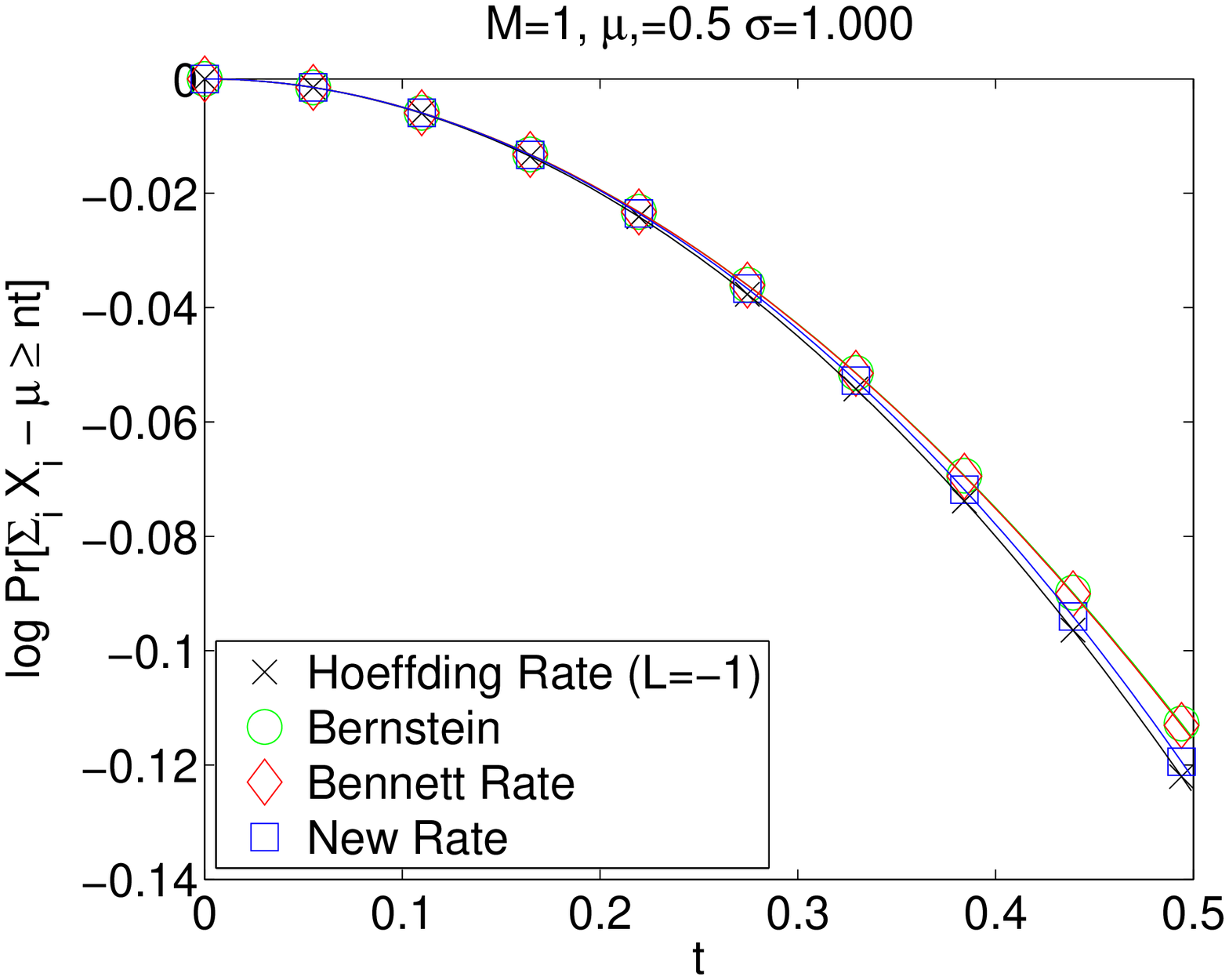} &
    \epsfxsize=2.46in
    \epsfbox{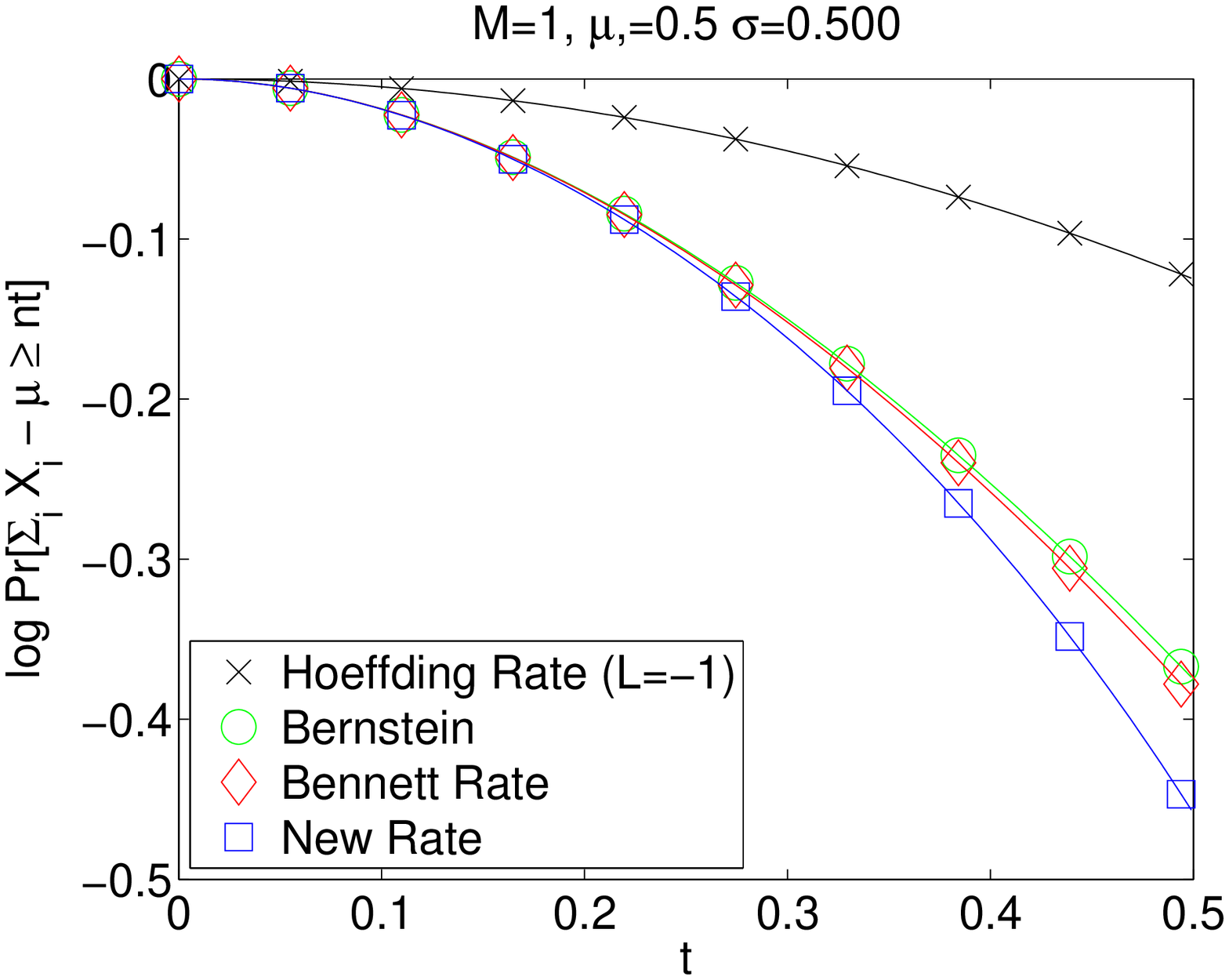} \\
    \epsfxsize=2.46in
    \epsfbox{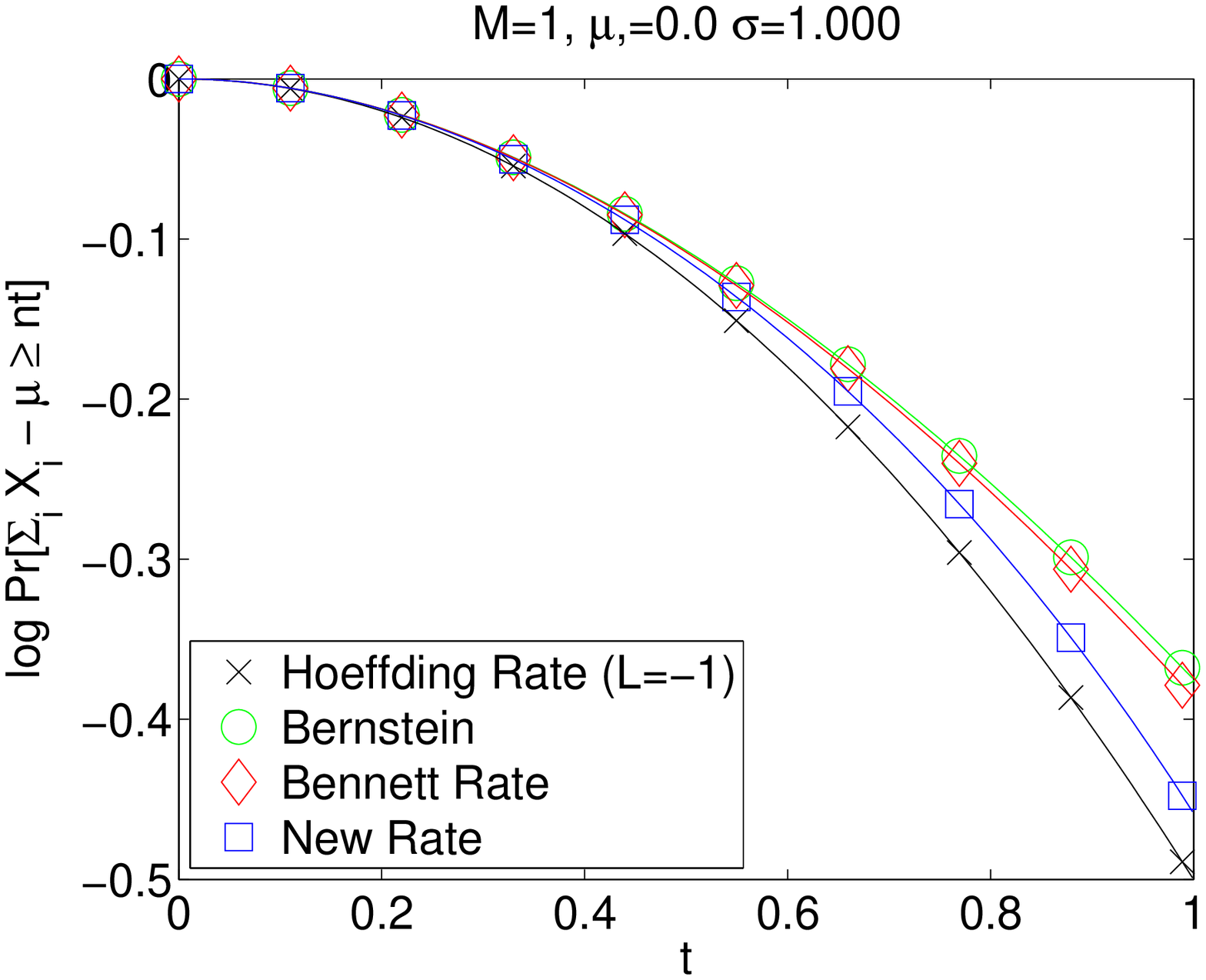} &
    \epsfxsize=2.46in
    \epsfbox{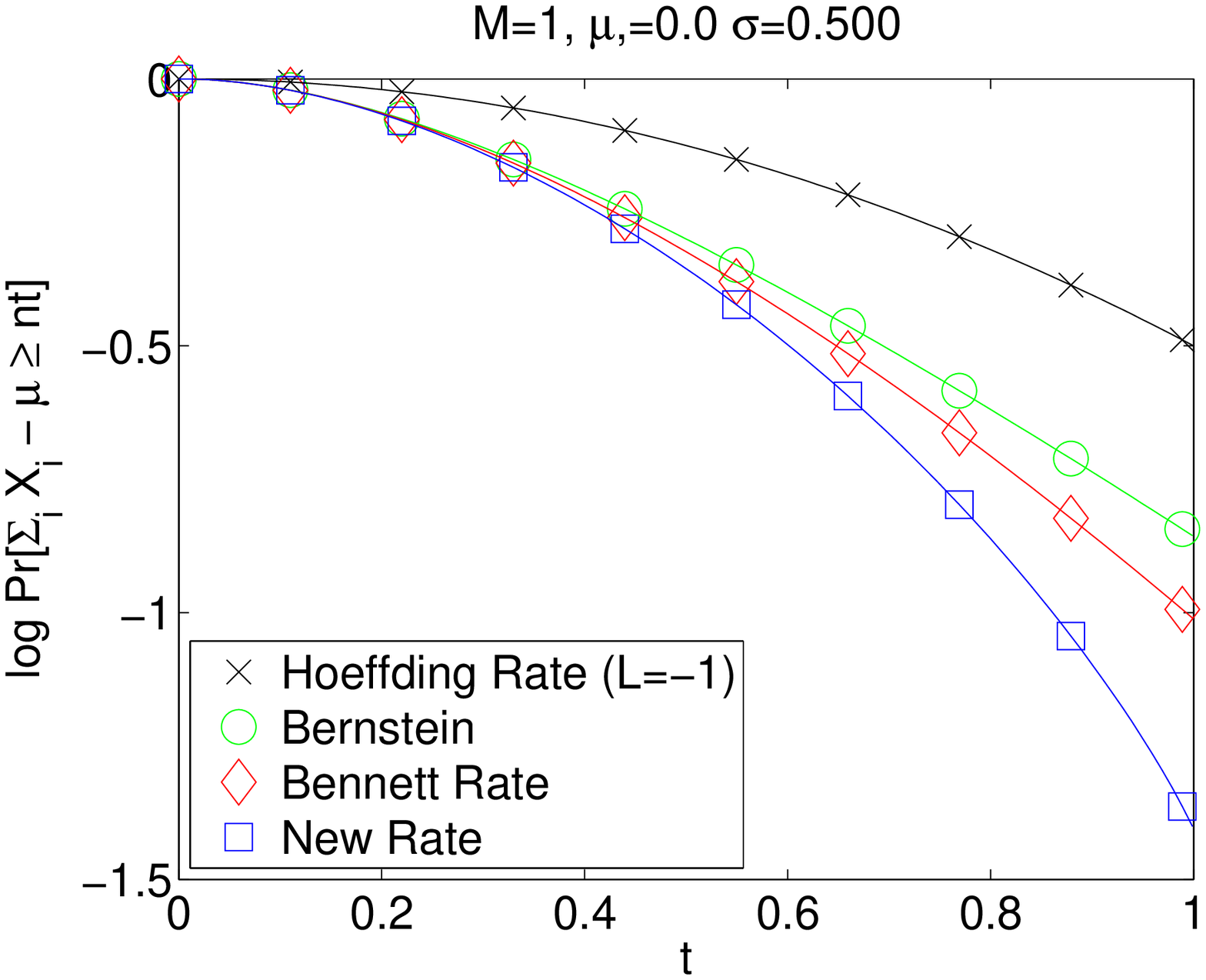} \\
    \epsfxsize=2.46in
    \epsfbox{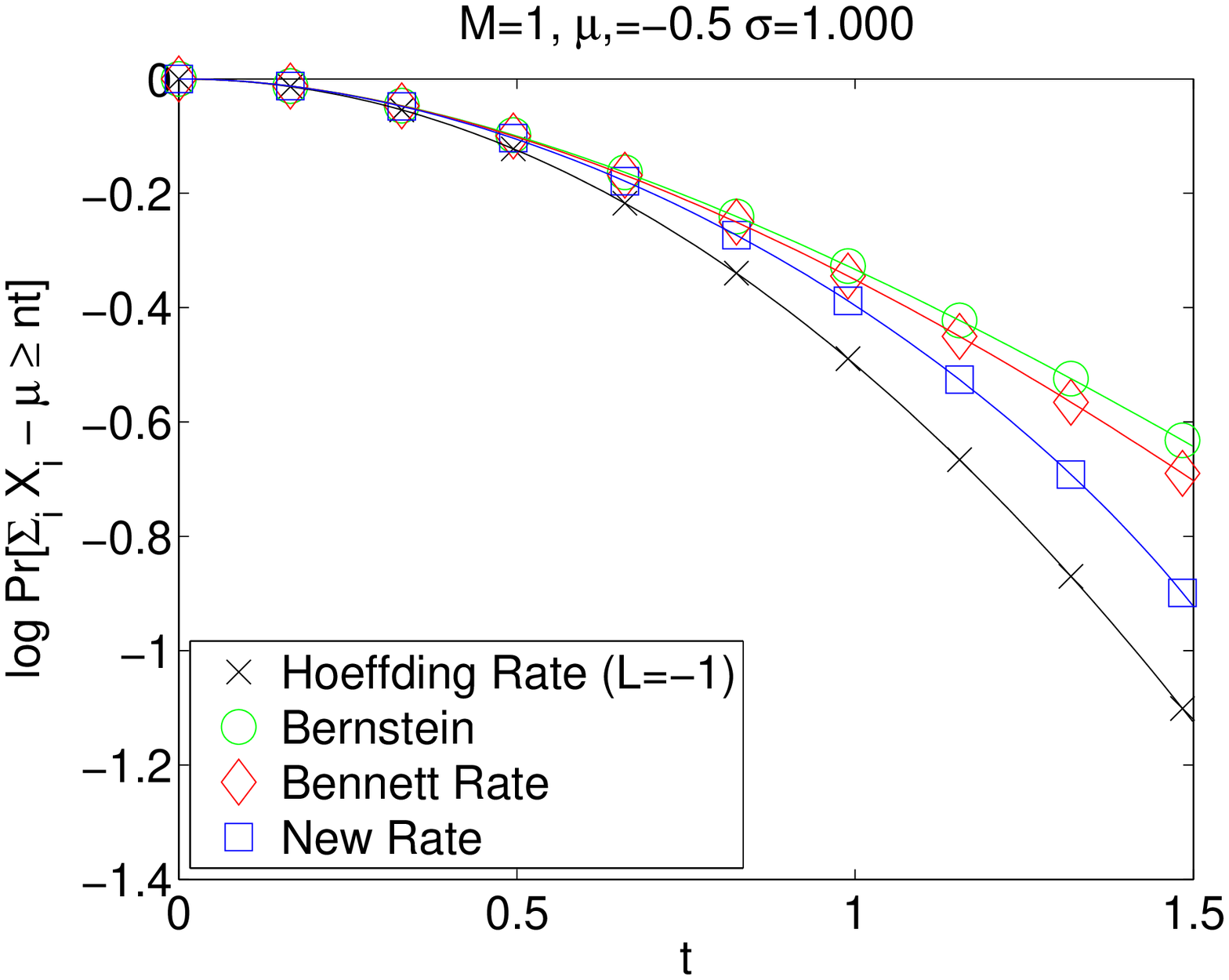} &
    \epsfxsize=2.46in
    \epsfbox{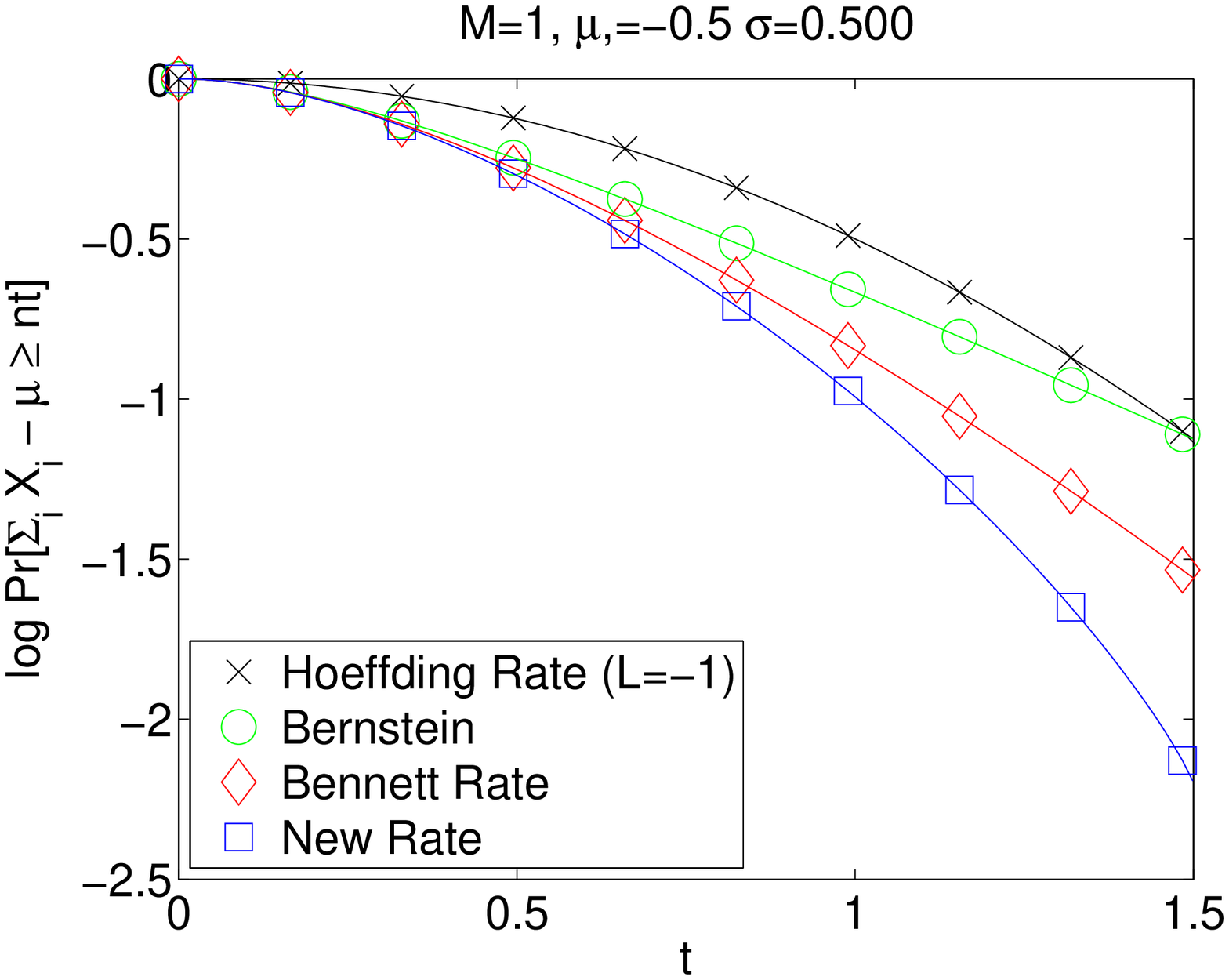} 
\end{tabular}
\caption[Comparing the four inequalities' convergence rates for homogeneous random variables under different $\mu$ and $\sigma$ settings. Lower values are better.]{Comparing the four inequalities' convergence rates for homogeneous random variables under different $\mu$ and $\sigma$ settings. Lower values are better.}
\label{fig:bounds1}
\end{figure}

\begin{figure}[htbp]
  \center
\setlength\tabcolsep{2pt}
  \begin{tabular}[b]{cc}
    \epsfxsize=2.46in
    \epsfbox{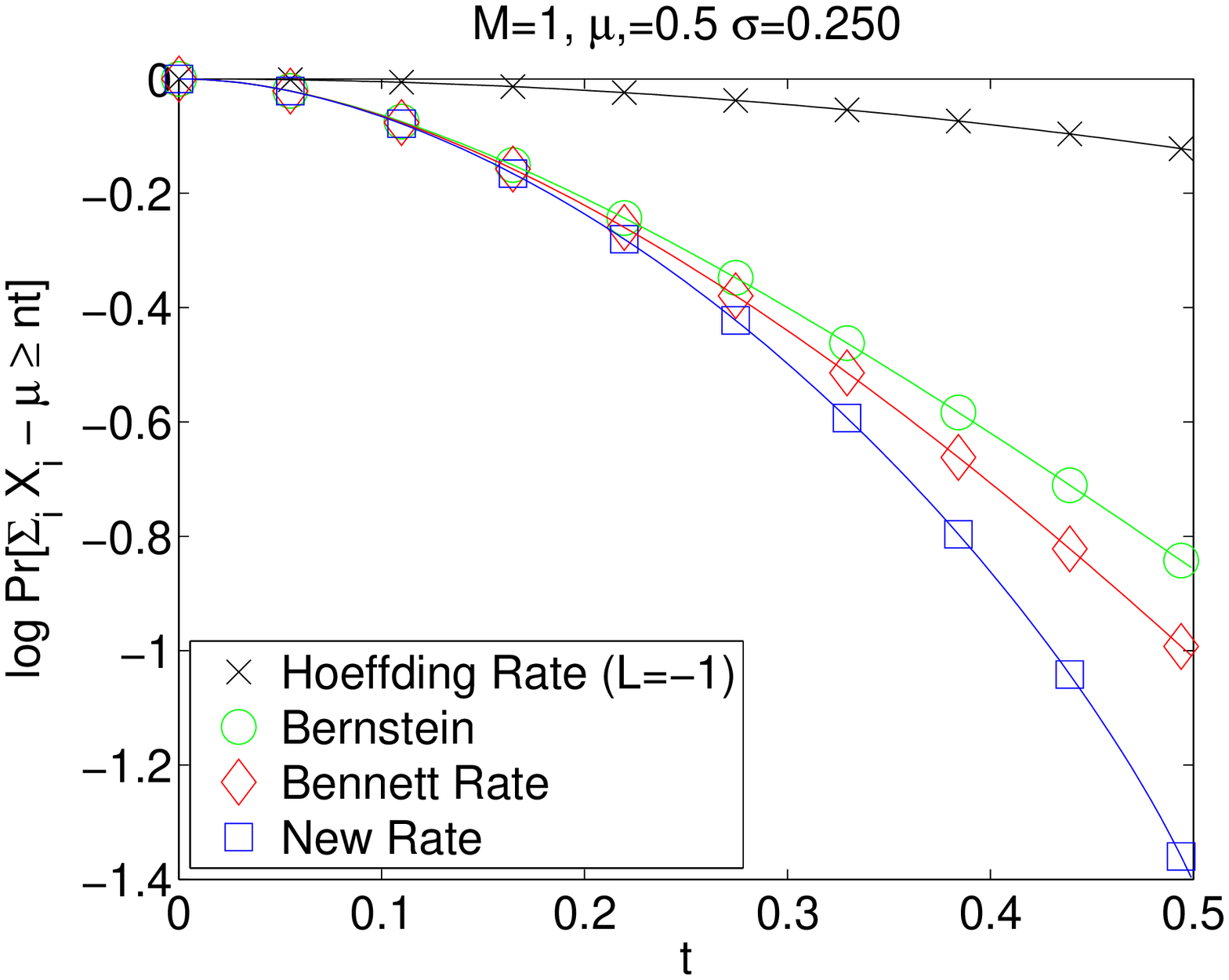} &
    \epsfxsize=2.46in
    \epsfbox{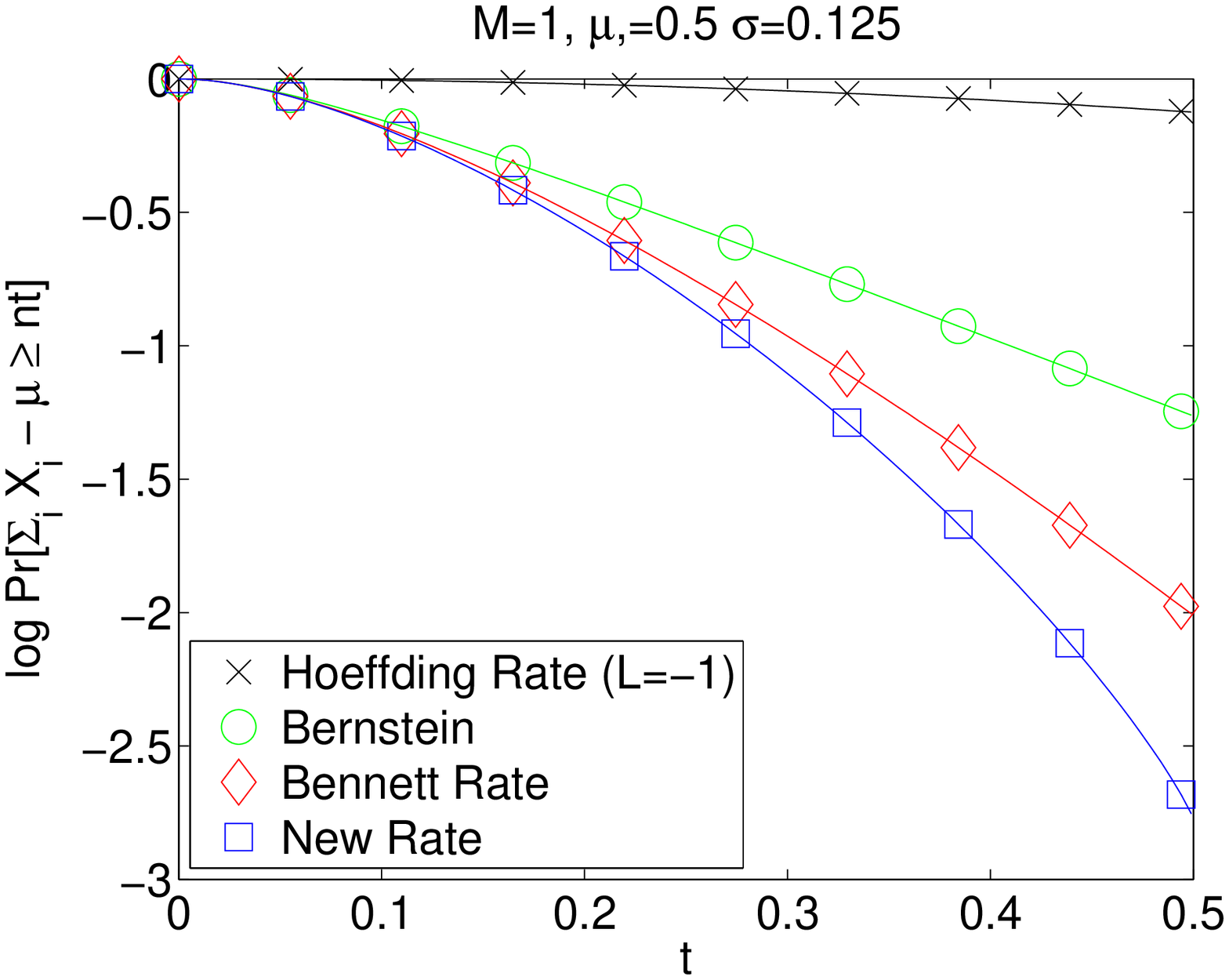} \\
    \epsfxsize=2.46in
    \epsfbox{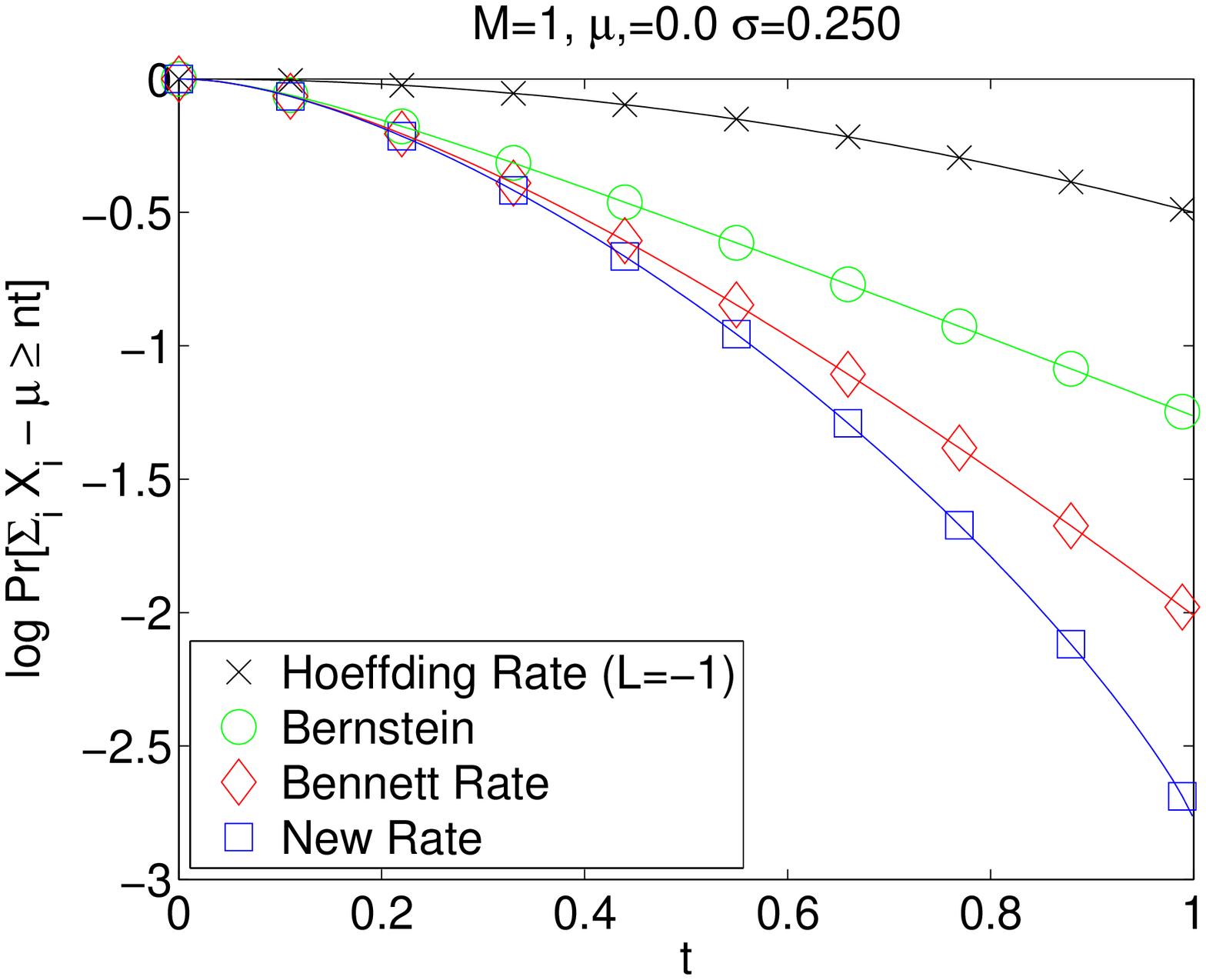} &
    \epsfxsize=2.46in
    \epsfbox{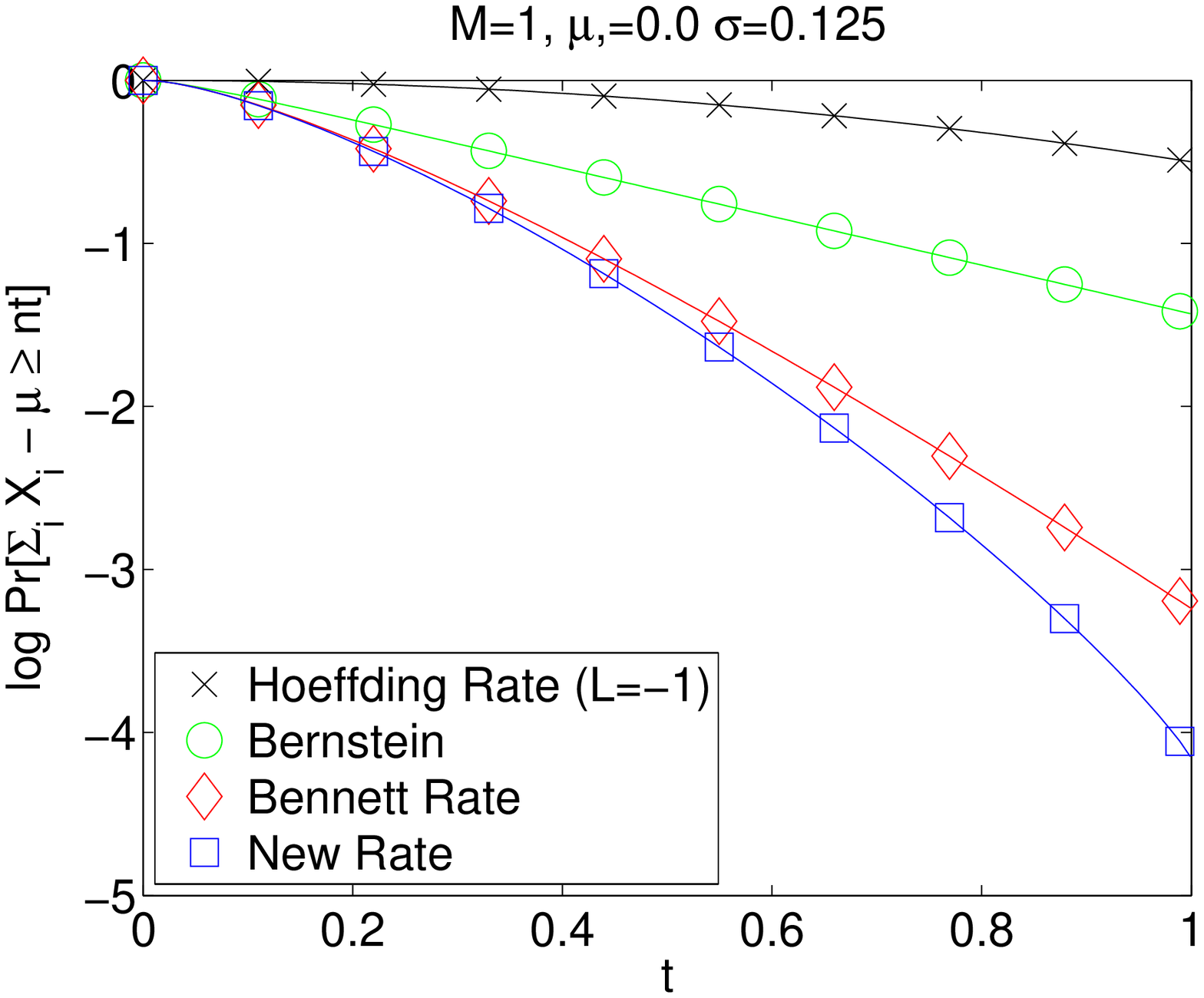} \\
    \epsfxsize=2.46in
    \epsfbox{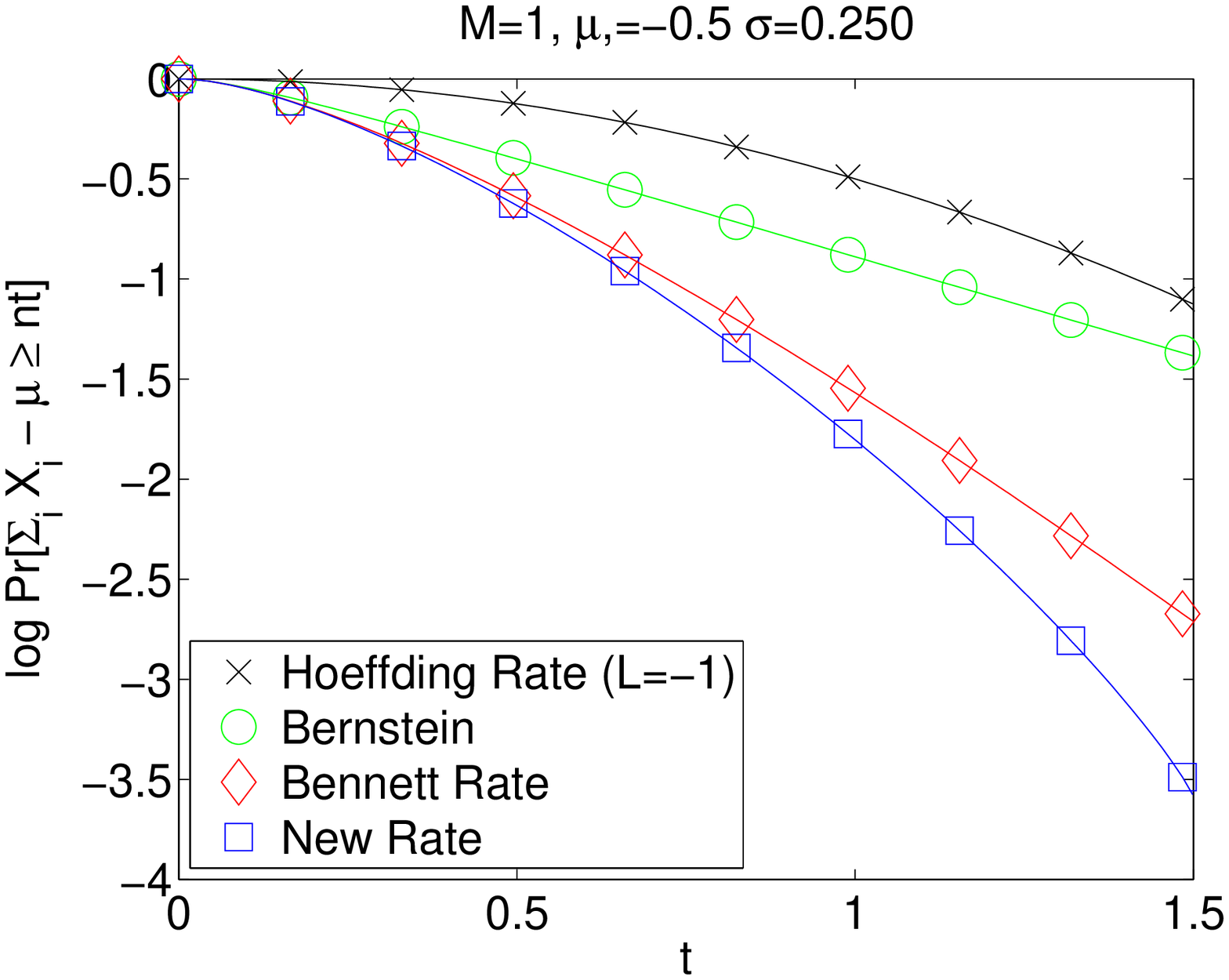} &
    \epsfxsize=2.46in
    \epsfbox{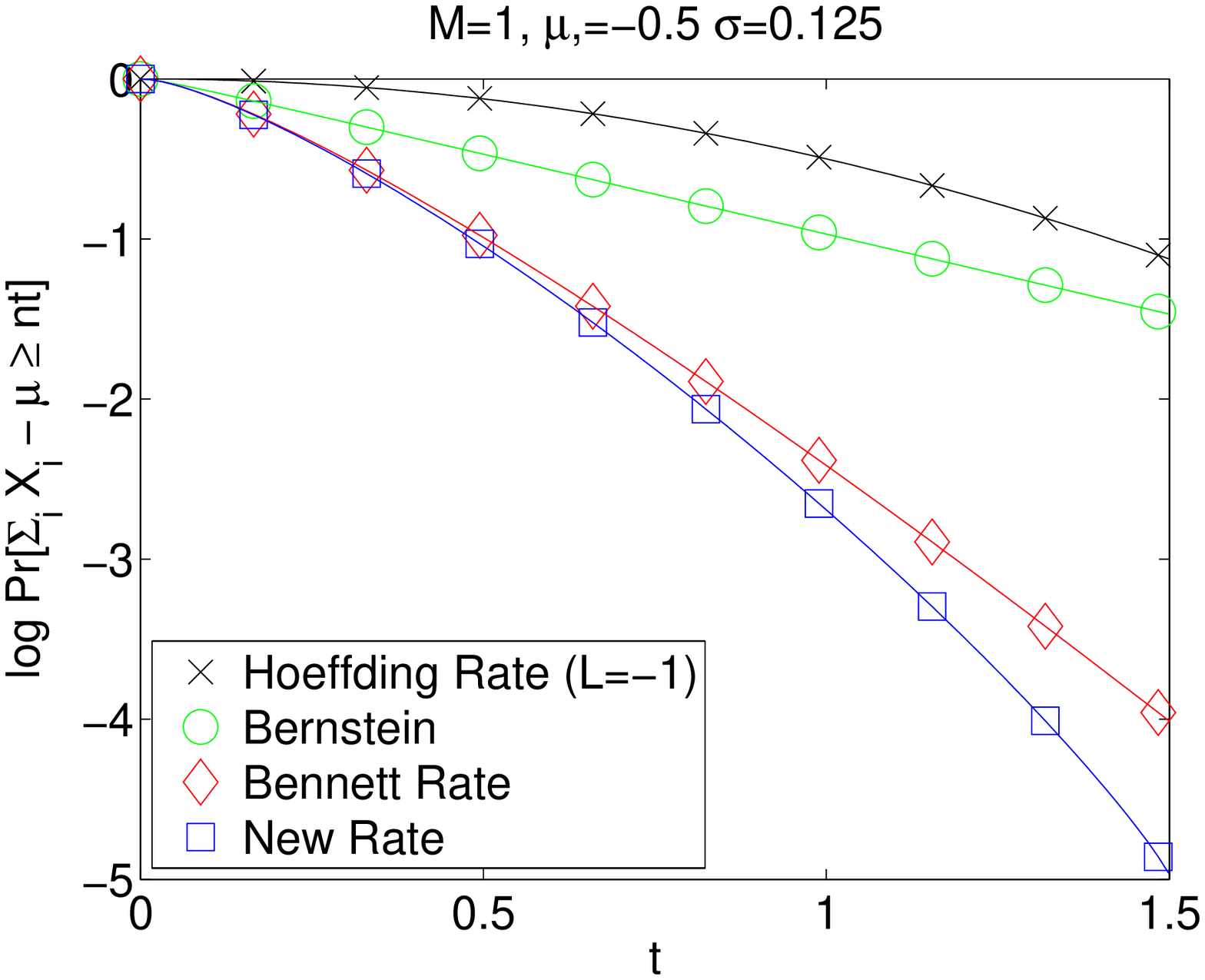} 
\end{tabular}
\caption[Comparing the four inequalities' convergence rates for homogeneous random variables under different $\mu$ and $\sigma$ settings. Lower values are better.]
{
Comparing the four inequalities' convergence rates for homogeneous random variables under different $\mu$ and $\sigma$ settings. Lower values are better.
}
\label{fig:bounds2}
\end{figure}

\begin{figure}[htbp]
  \center 
\setlength\tabcolsep{2pt}
  \begin{tabular}[b]{cc}
    \epsfxsize=2.46in
    \epsfbox{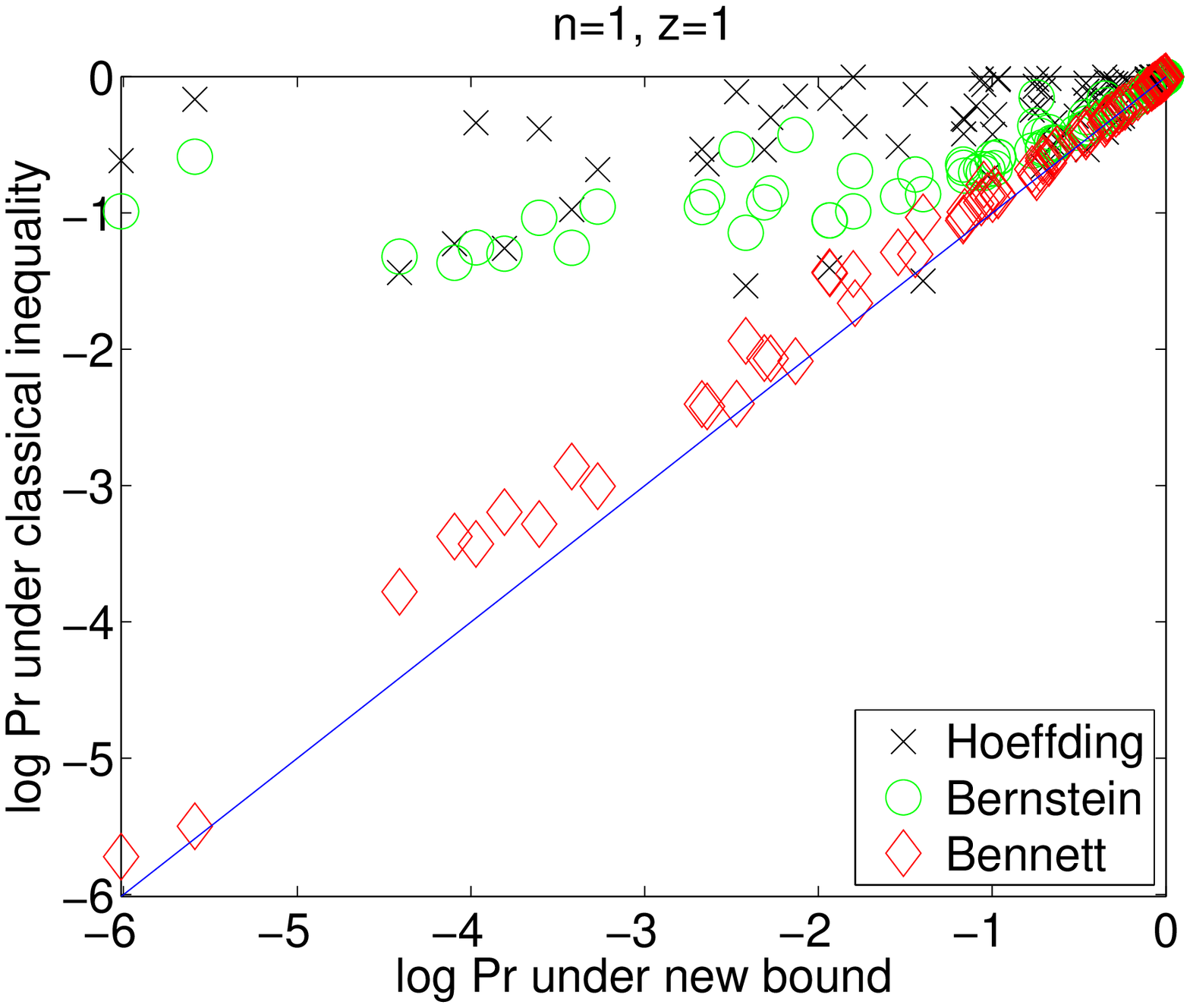} &
    \epsfxsize=2.46in
    \epsfbox{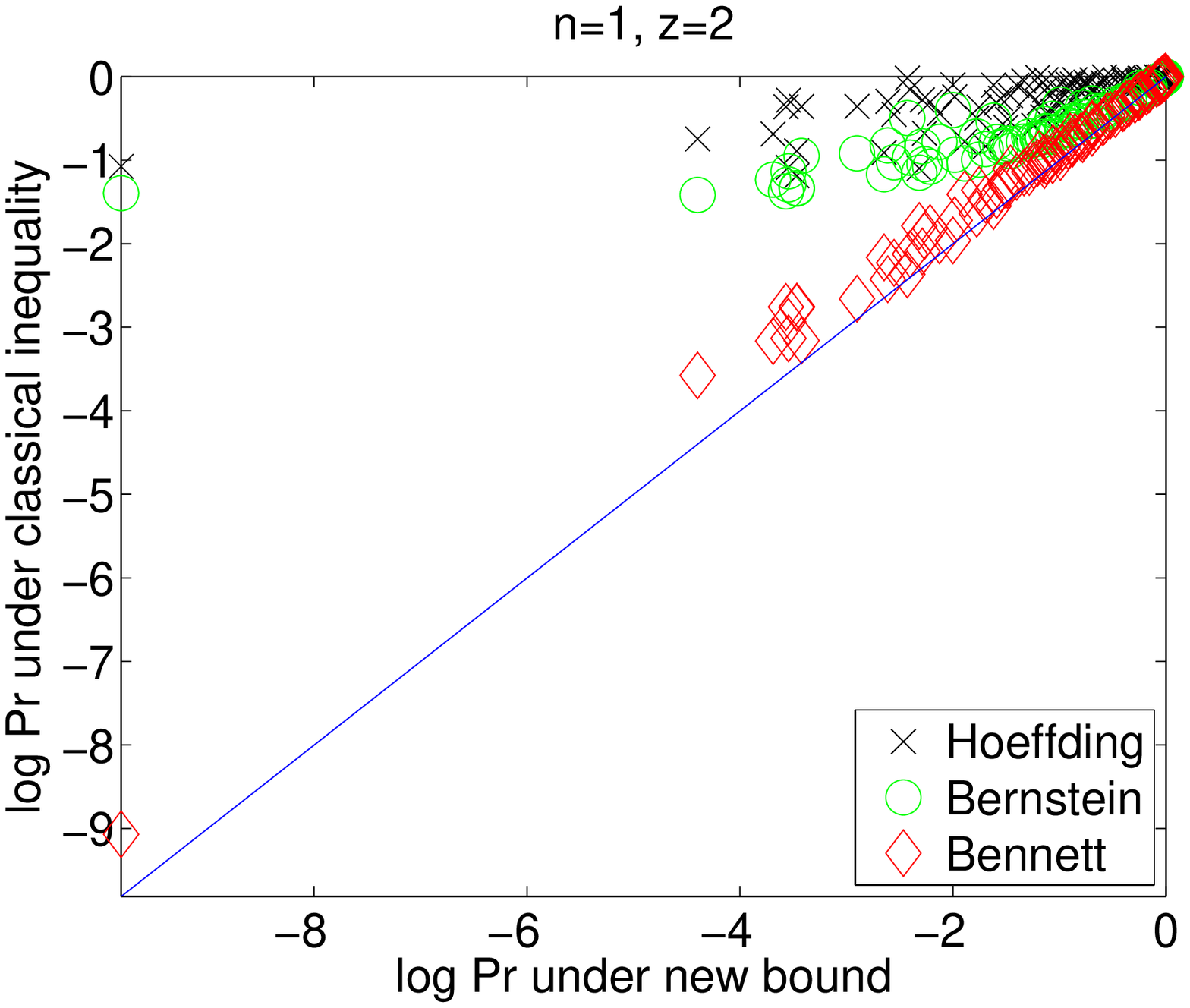} \\
    \epsfxsize=2.46in
    \epsfbox{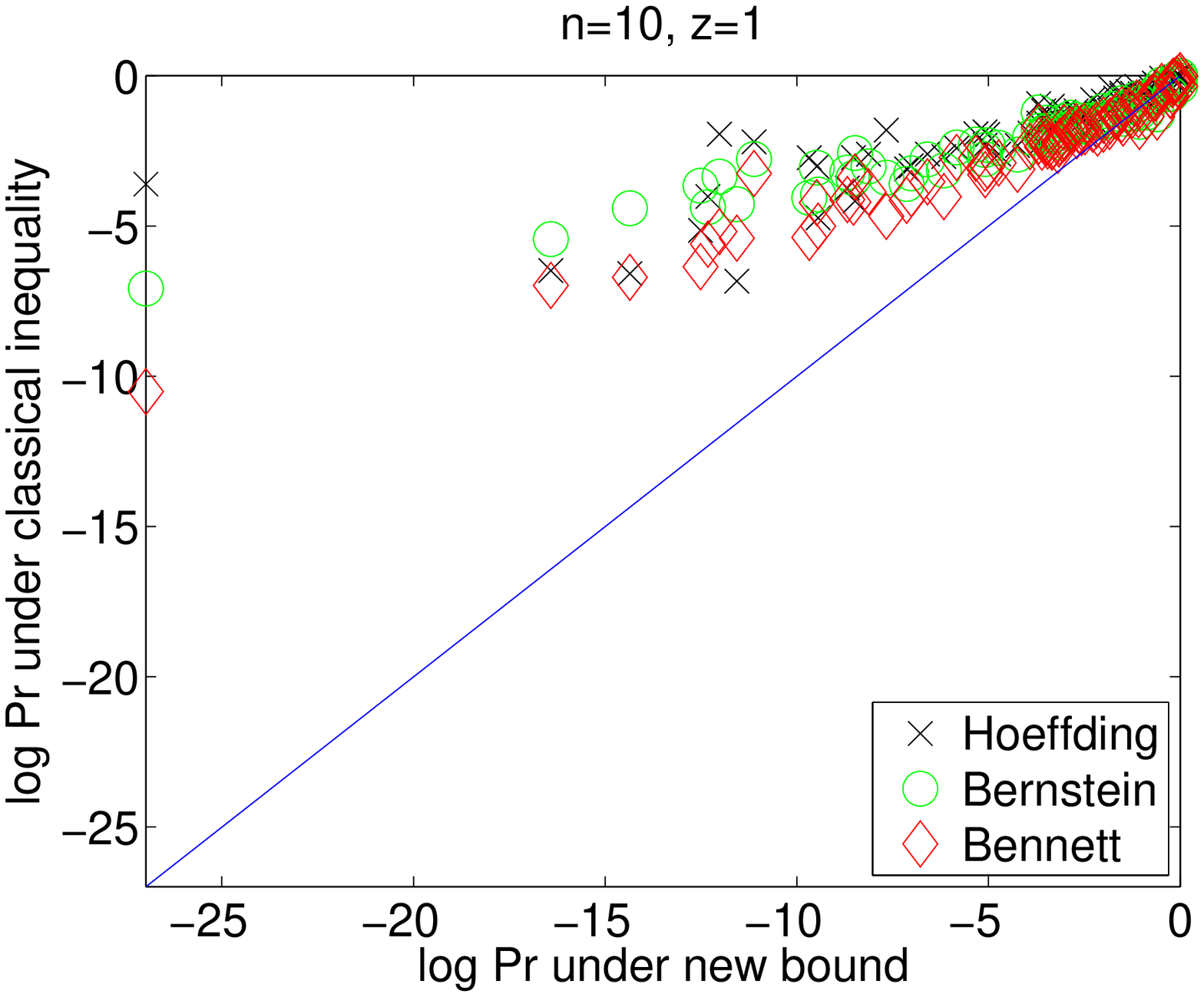} &
    \epsfxsize=2.46in
    \epsfbox{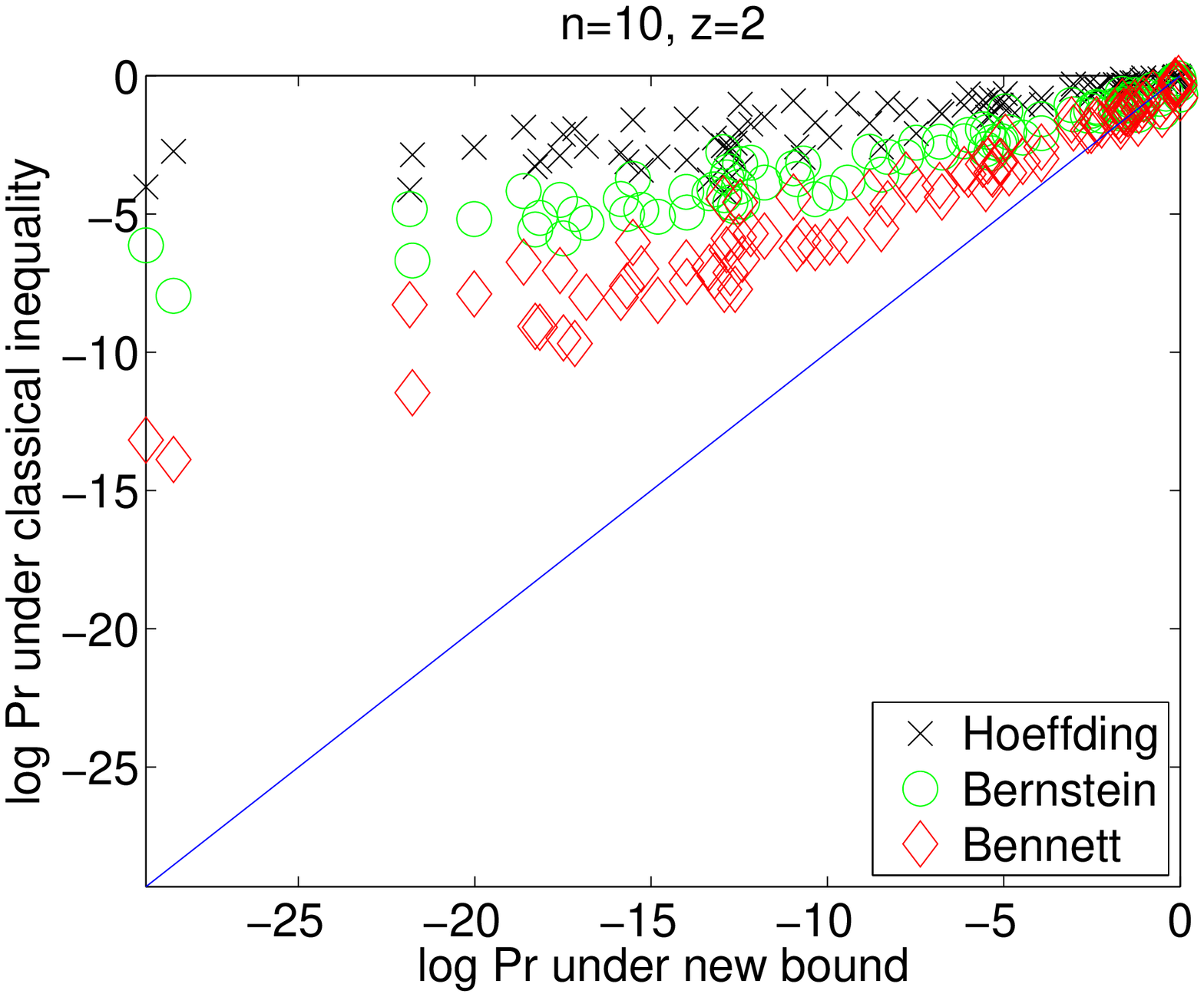} \\
    \epsfxsize=2.46in
    \epsfbox{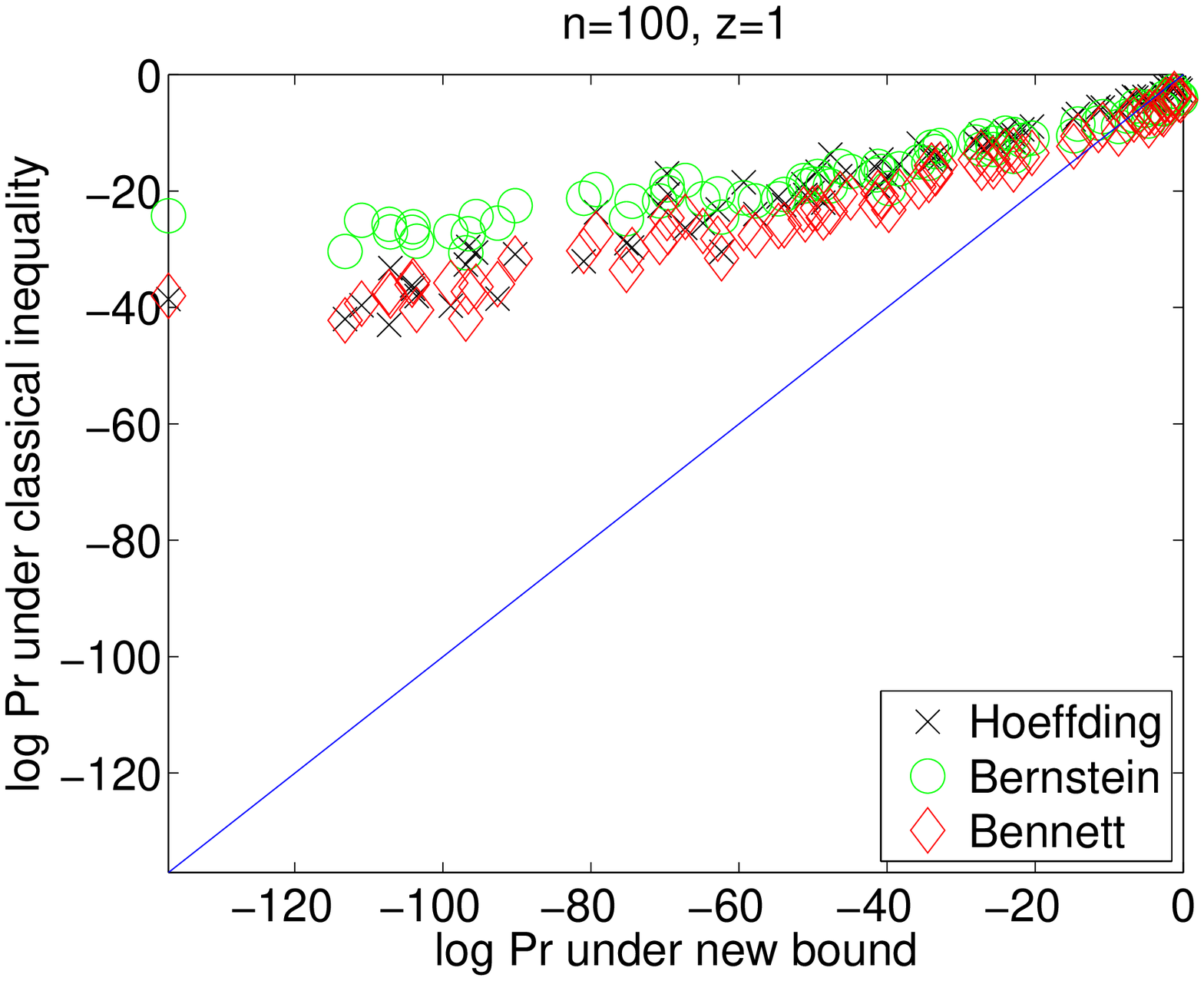} &
    \epsfxsize=2.46in
    \epsfbox{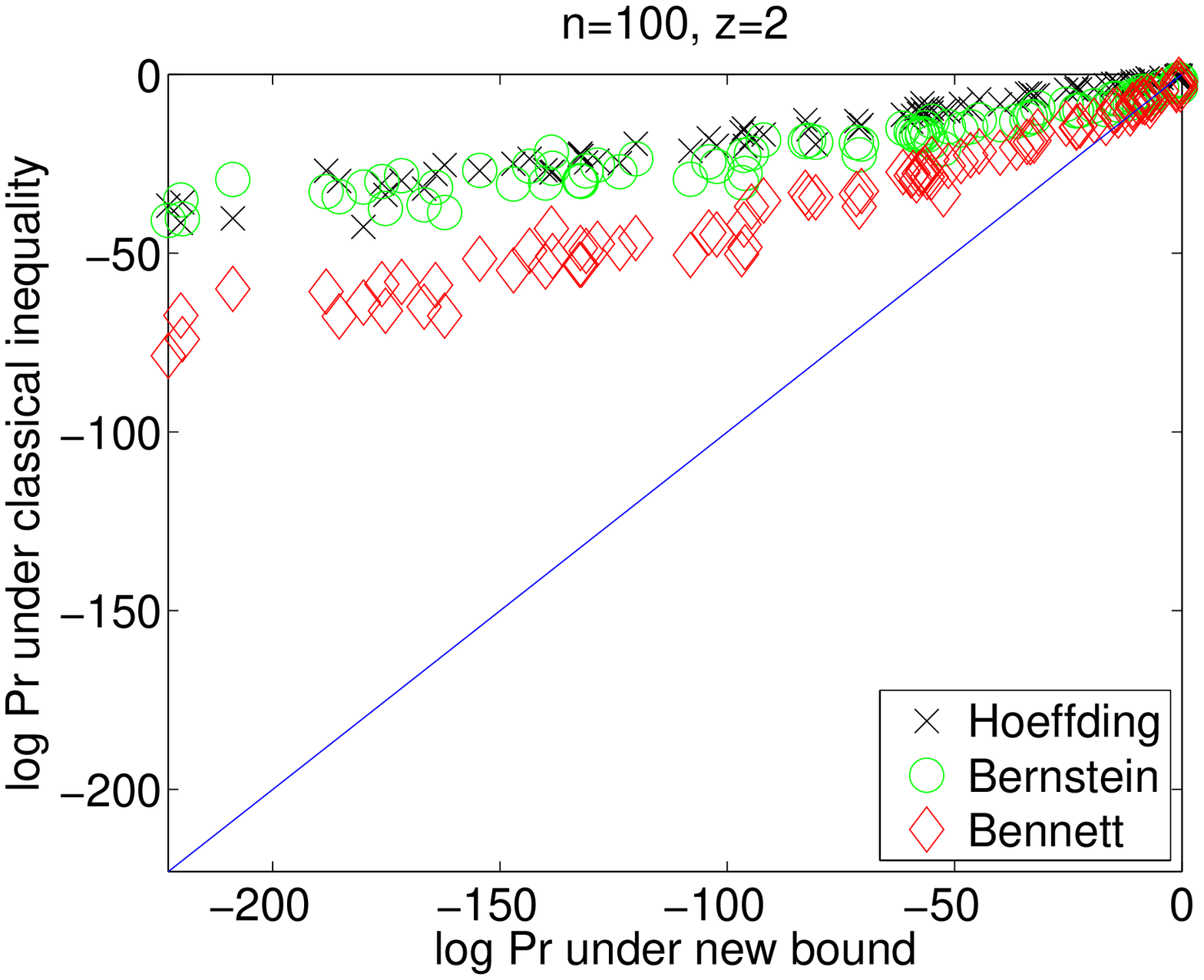} 
\end{tabular}
\caption[Comparing the classical inequalities with the new bound for heterogeneous random variables where the settings ($L_i$, $M_i$, $\mu_i$ and $\sigma_i$) and $t$ are drawn randomly. Specifically $\sigma_i$ is drawn from a uniform distribution over $(0,(M_i-L_i)/(2z))$. Various values of $n$ are also explored. A coordinate marker above the blue line indicates that the new bound is performing better than the classical inequality.]{Comparing the classical inequalities with the new bound for heterogeneous random variables where the settings ($L_i$, $M_i$, $\mu_i$ and $\sigma_i$) and $t$ are drawn randomly. Specifically $\sigma_i$ is drawn from a uniform distribution over $(0,(M_i-L_i)/(2z))$. Various values of $n$ are also explored. A coordinate marker above the blue line indicates that the new bound is performing better than the classical inequality.}
\label{fig:scatter1}
\end{figure}

\begin{figure}[htbp]
  \center
\setlength\tabcolsep{2pt}
  \begin{tabular}[b]{cc}
    \epsfxsize=2.46in
    \epsfbox{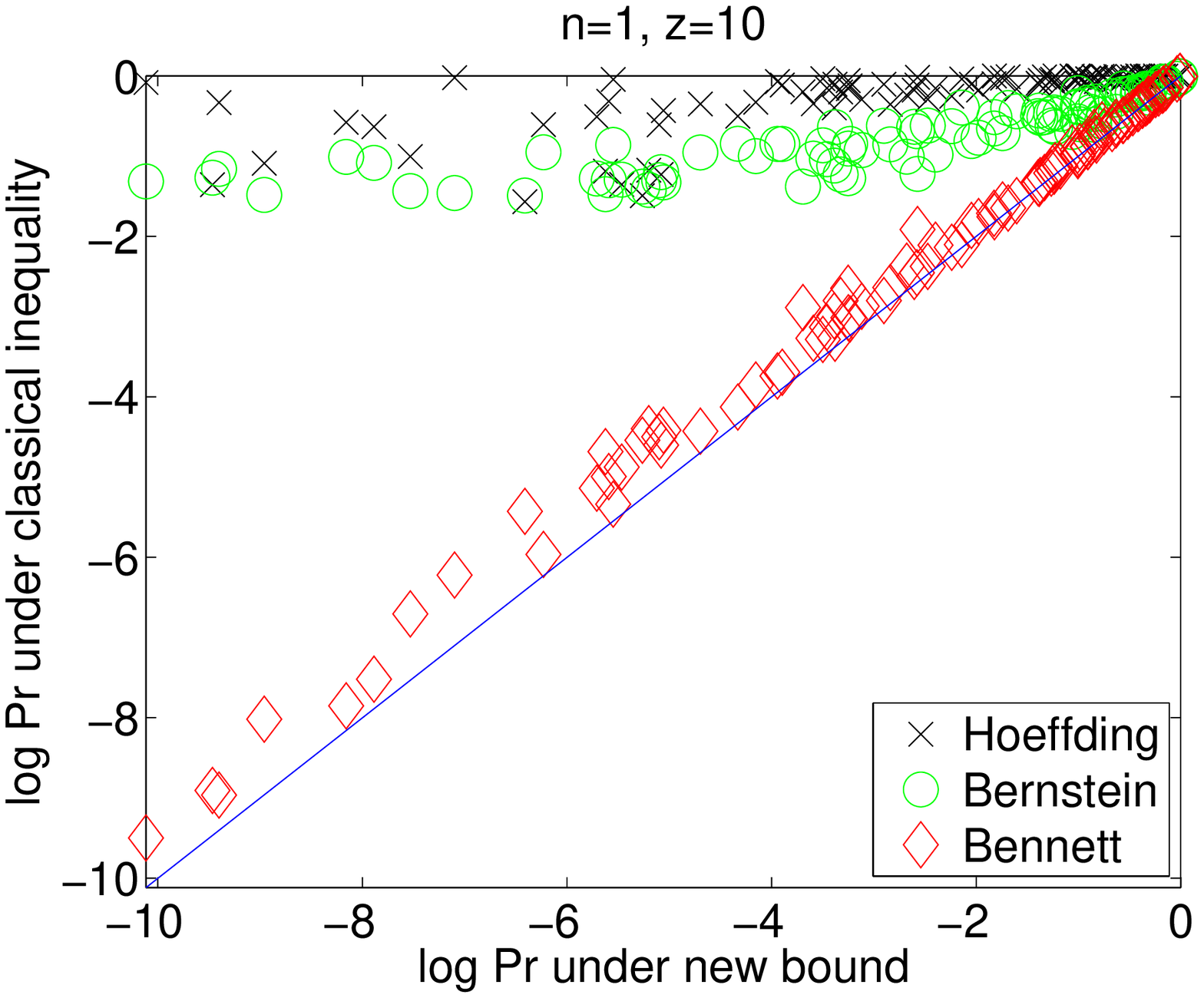} &
    \epsfxsize=2.46in
    \epsfbox{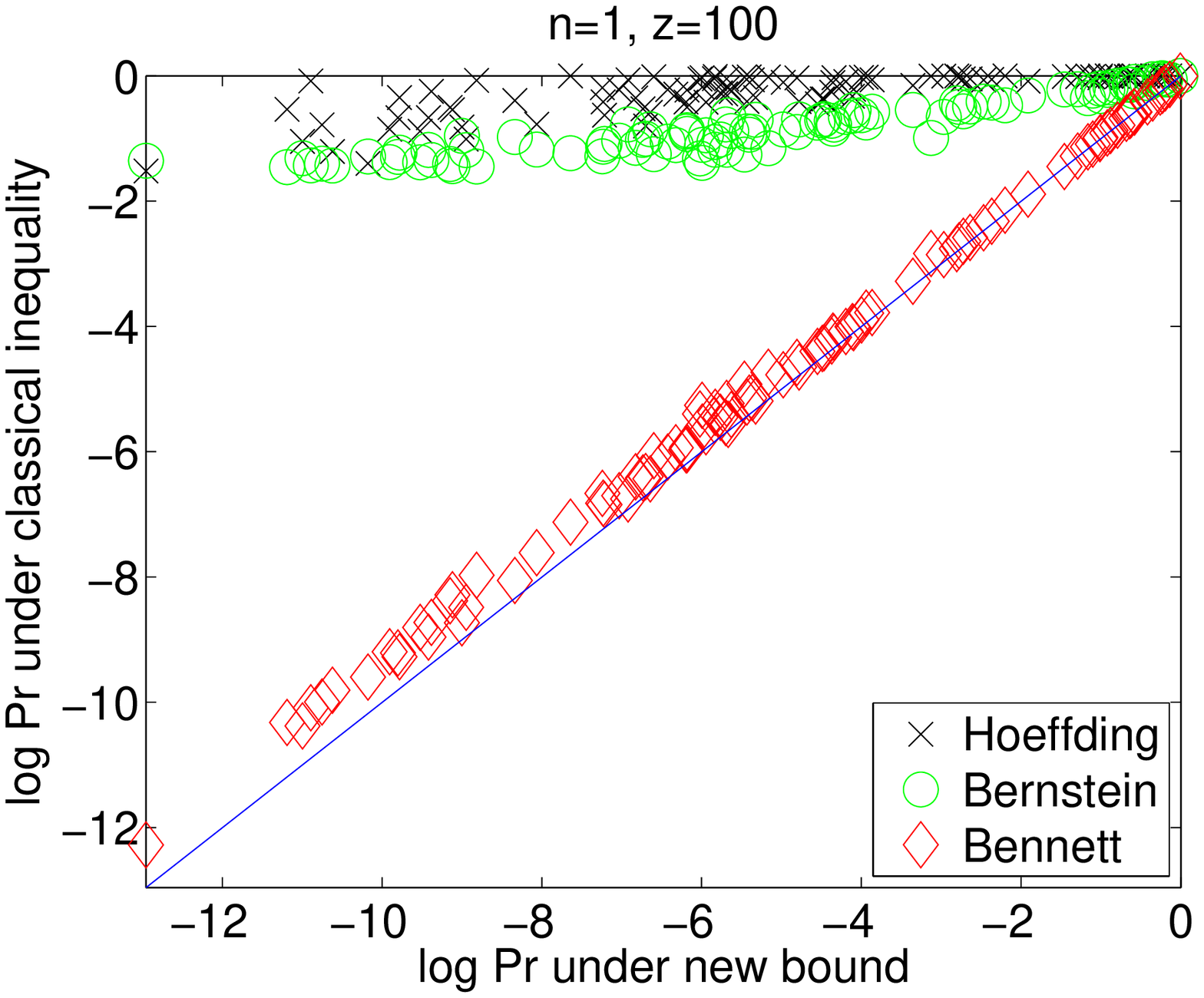} \\
    \epsfxsize=2.46in
    \epsfbox{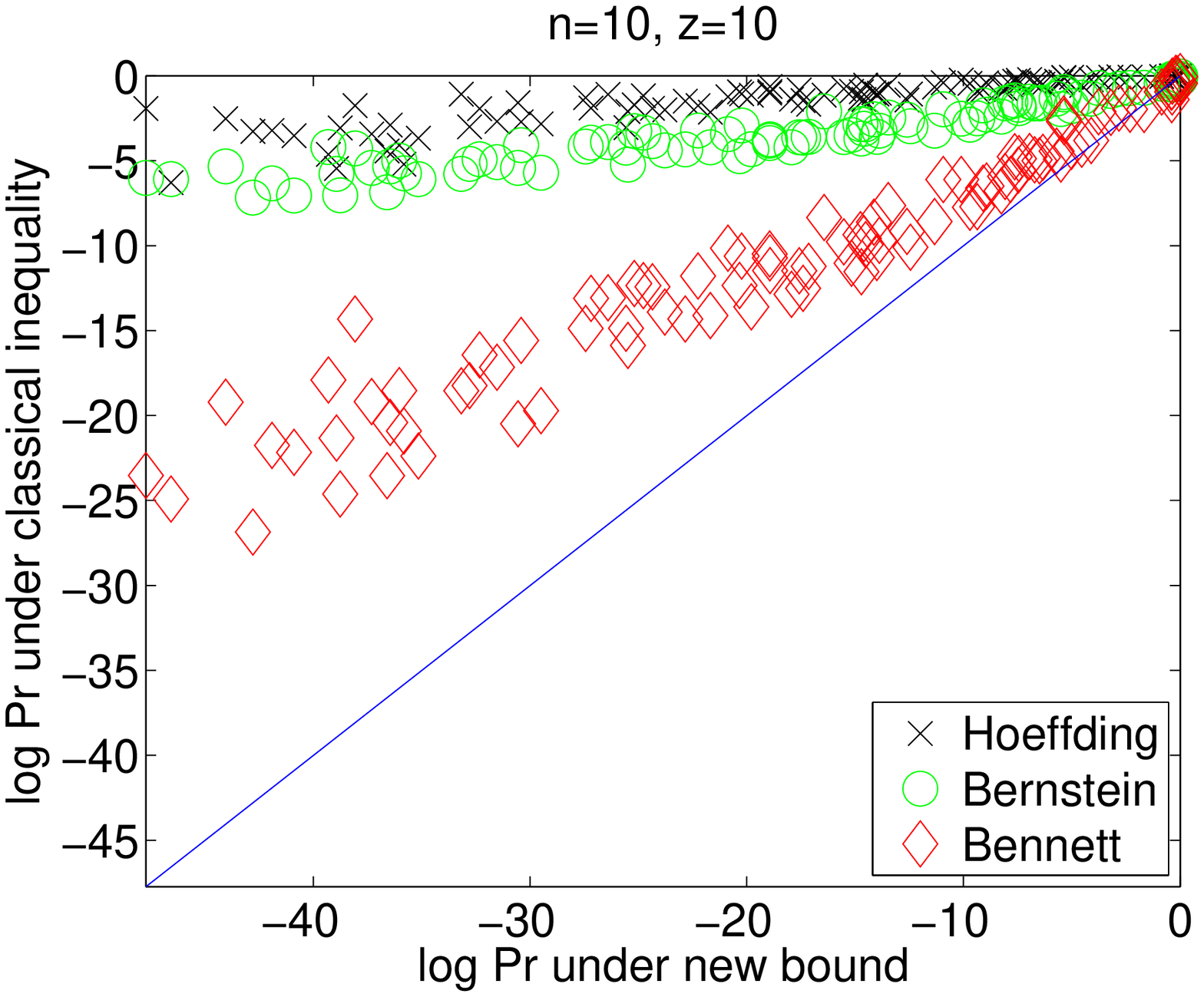} &
    \epsfxsize=2.46in
    \epsfbox{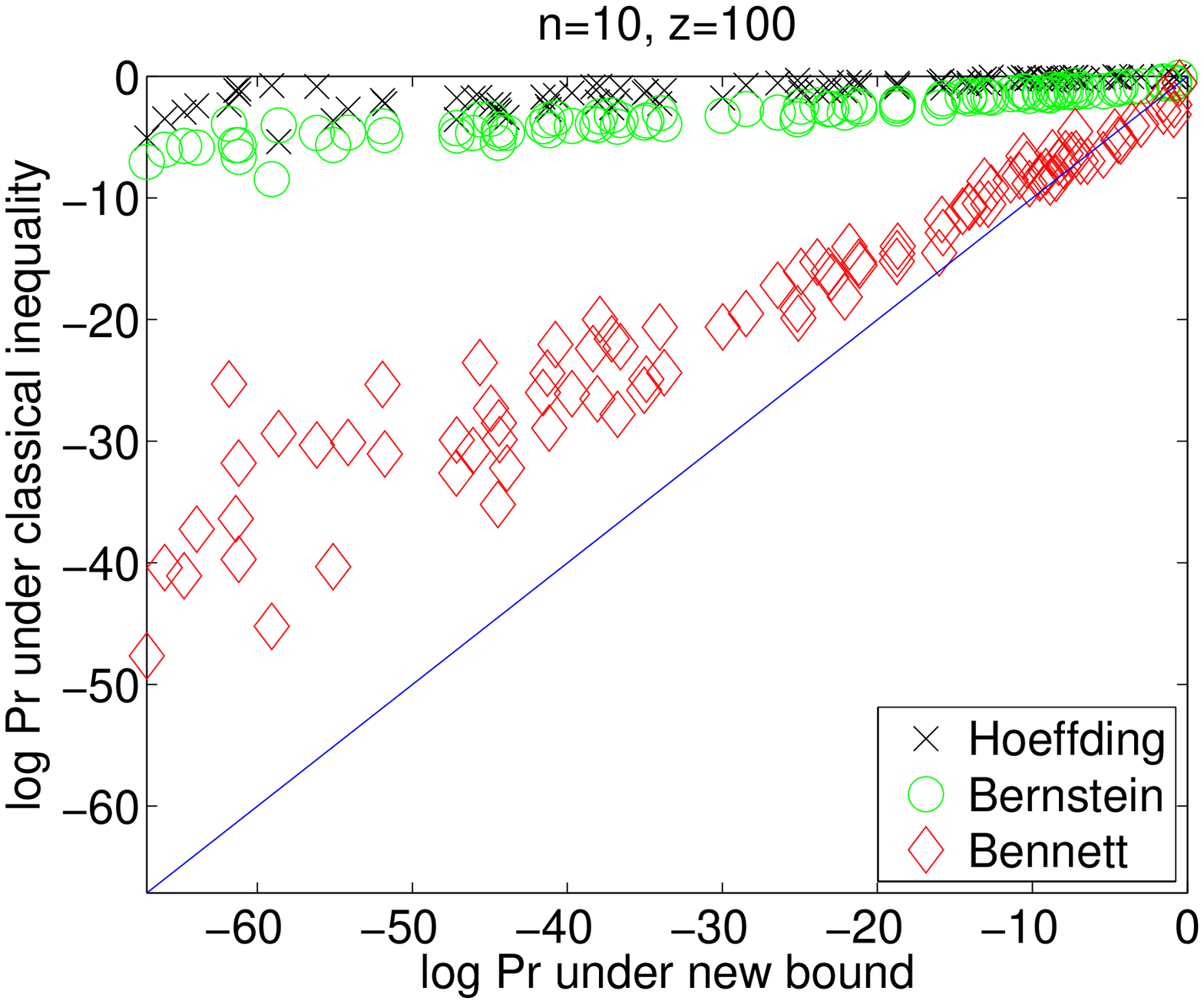} \\
    \epsfxsize=2.46in
    \epsfbox{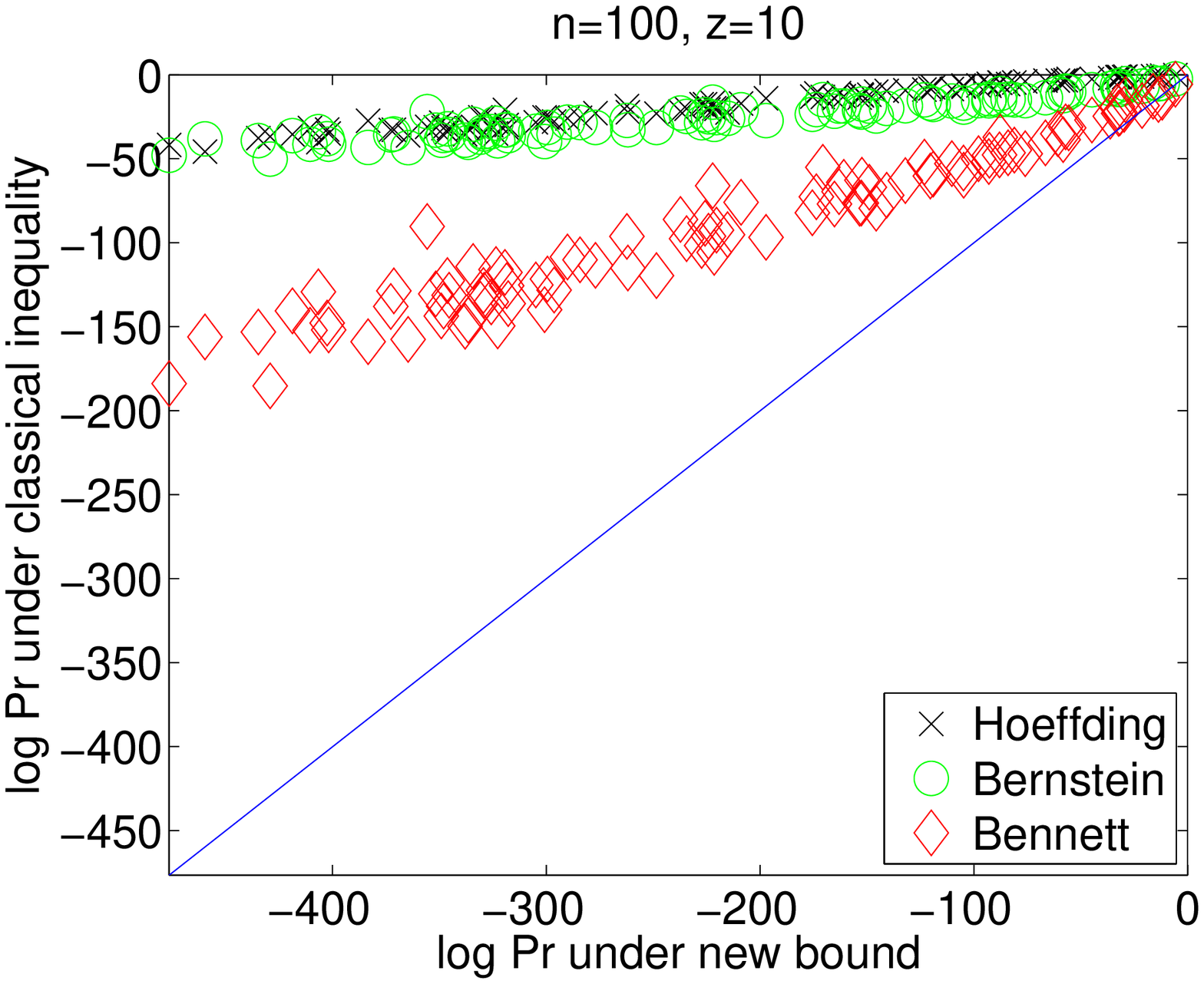} &
    \epsfxsize=2.46in
    \epsfbox{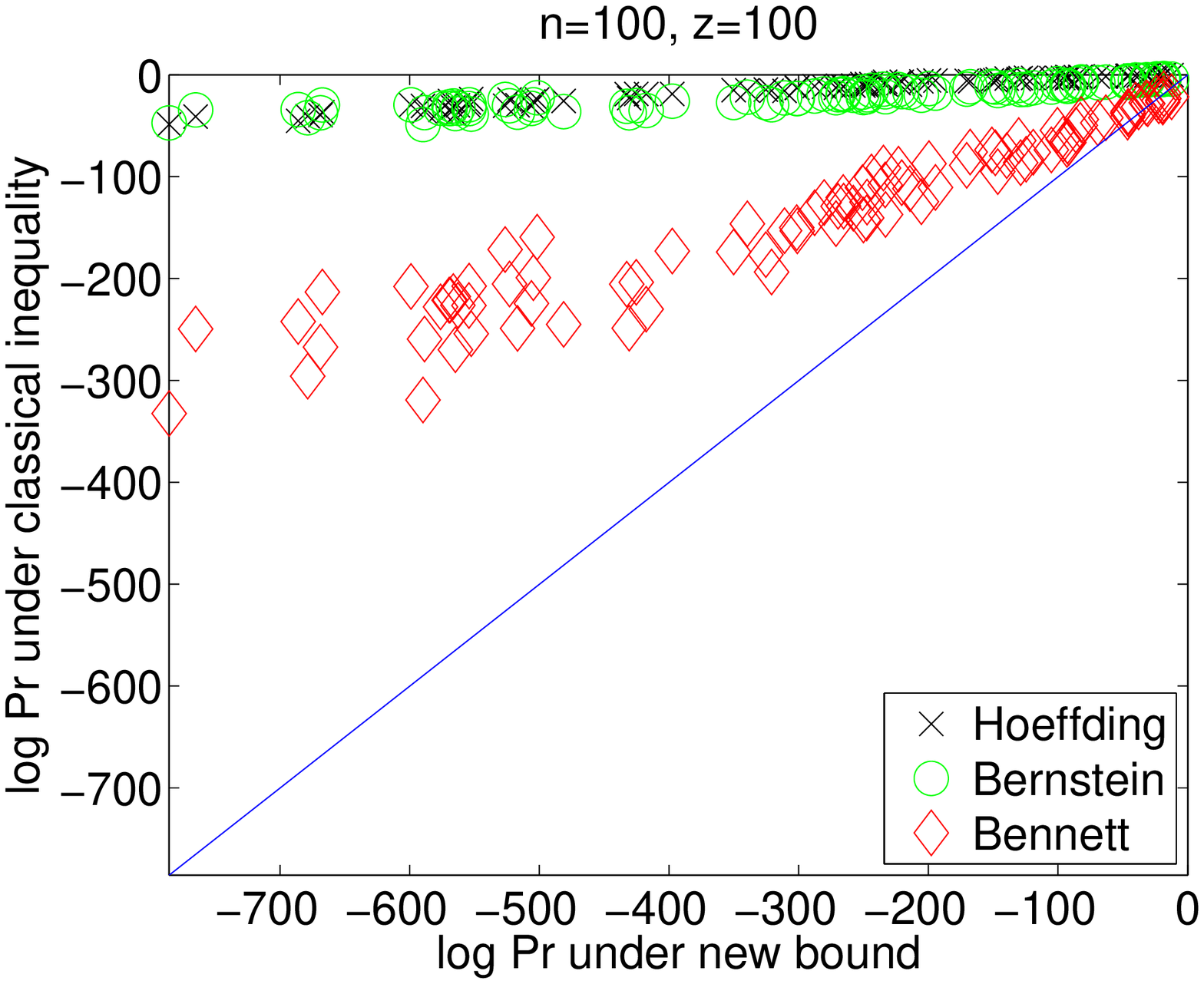}
\end{tabular}
\caption[Comparing the classical inequalities with the new bound for heterogeneous random variables where the settings ($L_i$, $M_i$, $\mu_i$ and $\sigma_i$) and $t$ are drawn randomly. Specifically $\sigma_i$ is drawn from a uniform distribution over $(0,(M_i-L_i)/(2z))$. Various values of $n$ are also explored. A coordinate marker above the blue line indicates that the new bound is performing better than the classical inequality.]{Comparing the classical inequalities with the new bound for heterogeneous random variables where the settings ($L_i$, $M_i$, $\mu_i$ and $\sigma_i$) and $t$ are drawn randomly. Specifically $\sigma_i$ is drawn from a uniform distribution over $(0,(M_i-L_i)/(2z))$. Various values of $n$ are also explored. A coordinate marker above the blue line indicates that the new bound is performing better than the classical inequality.}
\label{fig:scatter2}
\end{figure}

\begin{figure}[htbp]
  \center
\setlength\tabcolsep{2pt}
  \begin{tabular}[b]{c}
    \epsfxsize=3.8in
    \epsfbox{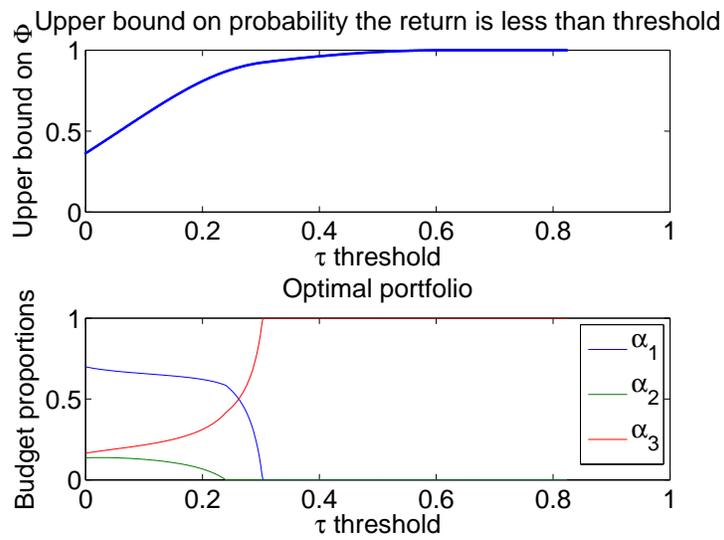}
\end{tabular}
\caption[The upper bound on the probability we will obtain a return
less than $\tau$ for 3 investments having a floor of 0 and
$\mu_1=0.3030, \mu_2=0.2400, \mu_3=0.6178$ with
$\sigma_1=0.2601,\sigma_2=0.5248,\sigma_3=0.7645$ (top panel).  The
corresponding optimal portfolio distribution is across $\tau$ values
(bottom panel).]{The upper bound on the probability we will obtain a
  return less than $\tau$ for 3 investments having a floor of 0 and
  $\mu_1=0.3030, \mu_2=0.2400, \mu_3=0.6178$ with
  $\sigma_1=0.2601,\sigma_2=0.5248,\sigma_3=0.7645$ (top panel).  The
  corresponding optimal portfolio distribution is plotted for a range
  of $\tau$ values (bottom panel).}
\label{fig:portfolio2}
\end{figure}

\begin{figure}[htbp]
  \center
\setlength\tabcolsep{2pt}
  \begin{tabular}[b]{c}
    \epsfxsize=3.8in
    \epsfbox{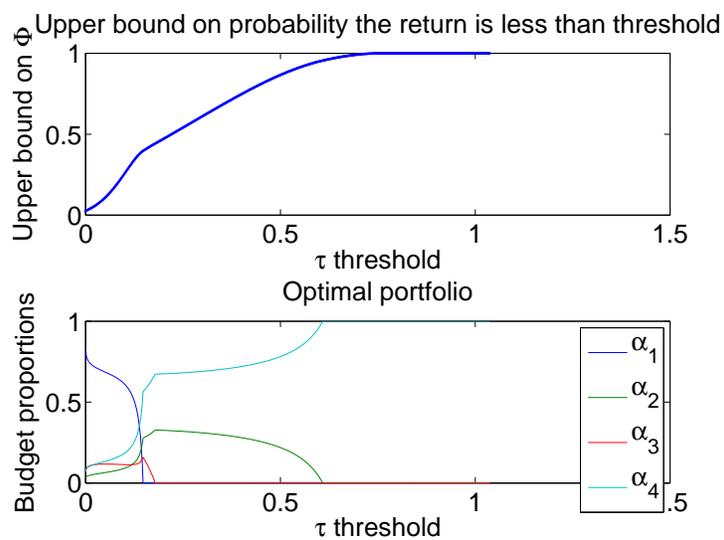}
\end{tabular}
\caption[The upper bound on the probability we will obtain a return
less than $\tau$ for 4 investments having a floor of 0 and
$\mu_1=0.1474, \mu_2=0.6088, \mu_3=0.1785, \mu_4=0.7585$ with
$\sigma_1=0.0593,\sigma_2=0.6218,\sigma_3=0.2183,\sigma_4=0.4597$ (top
panel). The corresponding optimal portfolio distribution is across
$\tau$ values (bottom panel).]{The upper bound on the probability we
  will obtain a return less than $\tau$ for 4 investments having a
  floor of 0 and $\mu_1=0.1474, \mu_2=0.6088, \mu_3=0.1785,
  \mu_4=0.7585$ with
  $\sigma_1=0.0593,\sigma_2=0.6218,\sigma_3=0.2183,\sigma_4=0.4597$
  (top panel). The corresponding optimal portfolio distribution is
  plotted for a range of $\tau$ values (bottom panel).}
\label{fig:portfolio1}
\end{figure}

\end{document}